\documentclass[letterpaper,12pt]{article}

\usepackage{arxiv}

\usepackage[utf8]{inputenc} 
\usepackage[T1]{fontenc}    
\usepackage[english]{babel}
\usepackage{hyphenat}
\usepackage{lmodern}
\usepackage{slantsc}
\usepackage{newtxtext}
\usepackage{libertine}
\usepackage{dsfont}
\usepackage{upgreek}
\usepackage{bold-extra}
\usepackage{pdfsync}
\usepackage{mathtools}

\mathtoolsset{showonlyrefs=true,showmanualtags=true} 
\usepackage{latexsym}
\usepackage[%
    pdftex,
    bookmarks=true,                                                 
    pagebackref=true,                                               
    pdfstartview={FitH}                                             
]{hyperref}                                                         
\renewcommand*\backref[1]{\ifx#1\relax \else [Cit. on p. #1]. \fi}
\usepackage{bookmark}
\usepackage{url}                                                    
\usepackage{booktabs}                                               
\usepackage{amsmath,amssymb,amsbsy,amsthm,amstext,amscd,amsfonts}   
\usepackage[nice]{nicefrac}                                         
\usepackage{microtype}                                              
\usepackage{lipsum}		                                            
\usepackage{mathscinet}
\usepackage[inline]{enumitem}
\usepackage{multirow}
\usepackage{graphicx}
\usepackage{float}
\usepackage{adjustbox}
\usepackage{caption}
\usepackage{subcaption}
\usepackage{color}
\usepackage[dvipsnames]{xcolor}
\usepackage{tikz}
\usepackage{array}
\usepackage[nottoc]{tocbibind}
\usepackage[square,sort,comma,numbers,sectionbib]{natbib}
\graphicspath{ {Figures/} }
\usepackage{doi}
\usepackage{indentfirst}
\usepackage{orcidlink}       
\usepackage{mathrsfs}
\usepackage{titlesec}
\usepackage{appendix}
\usepackage{algorithm}
\usepackage{algpseudocode}
\renewcommand{\gets}{\mathrel{\coloneqq}}

\numberwithin{equation}{section}

\theoremstyle{plain}
\newtheorem{theorem}{Theorem}[section]
\newtheorem*{theorem*}{Theorem}

\newtheorem{definition}[theorem]{Definition}
\newtheorem*{definition*}{Definition}
\newtheorem{lemma}[theorem]{Lemma}
\newtheorem*{lemma*}{Lemma}

\long\def\ie{\textit{i.e.}}
\long\def\eg{\textit{e.g.}}
\long\def\cf{\textit{cf.}}
\long\def\viz{\textit{viz.}}
\long\def\Prp{\emph{Problem} $\mathcal{P}_{\mathrm{ext}}$}

\makeatletter
\renewenvironment{proof}[1][\proofname]{%
	\par\pushQED{\qed}\normalfont%
	\topsep6\p@\@plus6\p@\relax
	\trivlist\item[\hskip\labelsep\bfseries#1\@addpunct{.}]%
	\ignorespaces
}{%
	\popQED\endtrivlist\@endpefalse
}
\makeatother

\newcommand{\bR}{\mathbb{R}}
\newcommand{\bC}{\mathbb{C}}
\newcommand{\bZ}{\mathbb{Z}}
\newcommand{\bN}{\mathbb{N}}

\newcommand{\cE}{\mathcal{E}}
\newcommand{\cV}{\mathcal{V}}
\newcommand{\cG}{\mathcal{G}}

\newcommand{\cH}{\mathcal{H}}

\newcommand{\cW}{\mathcal{W}}
\newcommand{\Rint}{\mathring{R}}
\newcommand{\Real}[1]{\mathfrak{Re}\left( #1 \right)}

\newcommand{\tps}[1]{\ensuremath{{#1}^{\top}}}

\def\articletitle{Propagation of Waves from Finite Sources Arranged in Line Segments within an Infinite Triangular Lattice}
\def\articleauthorD{David~Kapanadze}
\def\articleauthorV{Zurab~Vashakidze}

\title{\articletitle}

\date{} 					

\author{
        \hspace{1mm}{\articleauthorD}\,\orcidlink{0000-0002-5006-8615} \\
	A. Razmadze Mathematical Institute, TSU, \\ 
        Merab Aleksidze II Lane 2, Tbilisi 0193, Georgia \\[6pt]
	Free University of Tbilisi, \\
        240 David Aghmashenebeli Alley, Tbilisi 0159, Georgia \\
	\href{mailto:david.kapanadze@gmail.com}{\texttt{david.kapanadze@gmail.com}}\\
	\And
	\hspace{1mm}{\articleauthorV}\,\orcidlink{0000-0001-8736-6213} \\
	The University of Georgia (UG), \\
	77a M. Kostava St., Tbilisi 0171, Georgia \\[6pt]
	Ilia Vekua Institute of Applied Mathematics (VIAM), TSU, \\
	2 University St., Tbilisi 0186, Georgia \\
	\href{mailto:zurab.vashakidze@gmail.com}{\texttt{zurab.vashakidze@gmail.com}}, \href{mailto:z.vashakidze@ug.edu.ge}{\texttt{z.vashakidze@ug.edu.ge}}
}

\renewcommand{\shorttitle}{Wave Propagation from Line Sources in a Triangular Lattice}

\hypersetup{%
    pdftitle={\articletitle},
    pdfsubject={Analysis of PDEs (math.AP); Mathematical Physics (math-ph); Classical Analysis and ODEs (math.CA)},
    pdfauthor={\articleauthorD,~\articleauthorV},
    pdfkeywords={discrete Helmholtz equation, exterior Dirichlet problem, metamaterials, triangular lattice model},
    hidelinks,
    colorlinks=true,
    linkcolor=blue,
    citecolor=red,
    filecolor=cyan,
    urlcolor=teal,
}
\allowdisplaybreaks

\begin{document}
\maketitle
\vfill
\begin{abstract}
    This paper examines the propagation of time-harmonic waves in a two-dimensional triangular lattice with a lattice constant $a = 1$. The sources are positioned along line segments within the lattice. Specifically, we investigate the discrete Helmholtz equation with a wavenumber $k \in \left( 0,2\sqrt{2} \right)$, where input data is prescribed on finite rows or columns of lattice sites. We focus on two main questions: the efficacy of the numerical methods employed in evaluating the Green's function, and the necessity of the cone condition. Consistent with a continuum theory, we employ the notion of radiating solution and establish a unique solvability result and Green's representation formula using difference potentials. Finally, we propose a numerical computation method and demonstrate its efficiency through examples related to the propagation problems in the left-handed two-dimensional inductor-capacitor metamaterial.
\end{abstract}
\vfill


\section{Introduction}\label{sec:intro}
We are interested in methods for solving the Dirichlet-type boundary value problem for the discrete Helmholtz equation with uniform fixed mesh spacing. This equation commonly arises in lattice dynamics, where the difference equations can model wave propagation. Similar static issues can occur in various scenarios. For instance, a study of lattice defects can be found in Ref. \cite{PhysRevB.46.10613}. Lattice dynamics traditionally finds applications in the field of solid-state physics \cite{Maradudin1963}. One can study the propagation of waves through a periodic arrangement of interacting cells, where each cell has the same arrangement of interacting atoms. The propagation of waves through such a perfect lattice is a well-known topic \cite{BornHuang1998}. If there is a defect in the lattice, it will cause the waves to scatter. Wave propagation through discrete structures remains an active area of research today.

The triangular lattice is one of the five two-dimensional Bravais lattice types and appears naturally in applications, especially in crystals and materials with hexagonal symmetry \cite{BornHuang1998,burke1966origins}. For example, graphene, as observed in nature, is a two-dimensional material comprising a monolayer of carbon atoms arranged in a honeycomb structure, which can be described as a hexagonal lattice with a two-atom basis (see \cite{RevModPhys.81.109,RevModPhys.83.837}). According to the study presented in \cite{KapPes2024}, for certain configurations, problems defined on the hexagonal lattice can be equivalently reduced to analogous problems formulated on the triangular lattice. Additionally, one can consider two-dimensional passive propagation media, such as a host microstrip line network periodically loaded with series capacitors and shunt inductors for signal processing and filtering (as depicted in \hyperref[fig:frag_triang_lattice]{Figure \ref*{fig:frag_triang_lattice}}). This type of inductor-capacitor lattice is called a negative-refractive-index transmission-line (NRI-TL) metamaterial \cite{Caloz2005} or simply left-handed two-dimensional metamaterial. Suppose that monochromatic inputs are applied to finite rows/columns of lattice sites. Assume that the number of unit cells in this slab is large enough to make it prohibitively expensive to solve numerically for the voltage/current at every cell in the lattice until the system reaches a steady state. As a simplifying strategy, it can be expected that the limiting case, when the lattice is effectively infinite, is more flexible to analysis and provides a good approximation of the steady-state output at an exterior boundary.

The discrete Helmholtz equation is a fundamental equation in physics, engineering, and mathematics, widely employed in various scientific disciplines to describe wave propagation in discrete media. Its applicability extends to fields such as electromagnetics, acoustics, optics, and quantum mechanics. Additionally, discrete Helmholtz equations are closely related to discrete Schr\"{o}dinger equations, which naturally arise in the tight-binding model of electrons in crystals \cite{Bloch1929,Slater1954,harrison1989electronic}. Similar equations also arise in studies involving time-harmonic elastic waves in lattice models of crystals \cite{Brillouin1953SecondEd,Maradudin1963,Lifshitz1966}. Moreover, the Helmholtz equation is closely related to the Maxwell system (for time-harmonic fields); solutions of the scalar Helmholtz equation are used to generate solutions of the Maxwell system (Hertz potentials), and every component of the electric and magnetic field satisfies an equation of Helmholtz type (see Kirsch and Hettlich \cite{Kirsch2015}). Continuous models for the Helmholtz equation with smooth boundaries have been thoroughly studied by Colton and Kress \cite{ColtonKress2019} and in related works. However, the formulation of a discrete analogue of the Rayleigh-Sommerfeld scattering theory across different lattice structures remains an active area of research. Results pertaining to square and triangular lattices have been previously obtained in \cite{Kapanadze2018,KapanadzePesetskaya2023}. 

This paper extends our investigation of the discrete Helmholtz equation and its associated exterior problems in a two-dimensional triangular lattice with a lattice constant $a = 1$. We concentrate on two central questions: the efficacy of the numerical technique employed to compute the lattice Green's function and the requirement of the cone condition. To this end, the effect of finite sources arranged on line segments within an infinite triangular lattice has been studied. Mathematical modelling of the propagation problem under consideration leads us to study an exterior problem for a two-dimensional discrete Helmholtz equation with Dirichlet boundary conditions. We use the results from \cite{KaPes2023} and carry out our investigation without passing to the complex wavenumber. Specifically, we employ the radiation conditions and asymptotic estimates from \cite{Kapanadze2021}, alongside the Rellich-Vekua type theorem stated by Isozaki and Morioka \cite{IsozakiMorioka2014}. It is worth noting that these problems can be explored through the limiting absorption principle, {\cf} {\eg} Sharma in Ref. \cite{Sharma2016}. This work delves into the diffraction phenomena occurring in triangular and hexagonal lattices caused by finite cracks and rigid constraints. We establish the existence and uniqueness of solutions for the problem under consideration within the interval $k \in \left( 0,2\sqrt{2} \right)$. Moreover, we derive Green's representation formula for the solution using difference potentials. The proposed approach aims to simplify the solution of the prescribed boundary value problem by solving a linear system of boundary equations. Furthermore, the solution to the original problem is represented by a linear combination of lattice Green's functions and the solution derived from a reduced linear system of boundary equations. The precision of our solution is mostly linked to the accurate computation of lattice Green's functions. To achieve this objective, instead of using the numerical quadrature technique for computing lattice Green's functions (due to the rapid oscillatory behaviour of these functions), we adopt the approach proposed by Berciu and Cook \cite{Berciu2010}. This method enables the computation of Green's functions through elementary operations alone, thus avoiding the need for integration, and yielding physically consistent solutions. For numerical illustration, two sample problems are considered for $k = 2$. The first example features four discrete boundary points aligned along a single line, while the second comprises ten discrete boundary points situated on parallel line segments. Numerical results are detailed in \hyperref[sec:numresalts]{Section \ref*{sec:numresalts}}.

\section{Formulation of the problem}\label{sec:problem_form}
In accordance with customary mathematical notation, we denote the sets of integers, positive integers, real numbers, and complex numbers by the symbols $\bZ$, $\bN$, $\bR$, and $\bC$, respectively.

Let $\left\{ \cV,\cE \right\}$ be a periodic simple graph that defines a two-dimensional infinite triangular lattice $\mathfrak{T}$, where
\begin{equation}\label{eq:Vmain}
    \cV = \left\{ T\left( x_1,x_2 \right) \in \mathbb{R}^{2} : \left( x_1,x_2 \right) \in {\bZ}^{2} =\bZ \times \bZ \right\}\,,
\end{equation}
denotes the vertex set, while the edge set is denoted by $\mathcal{E}$. The endpoints $\left( v,w \right) \in \cV \times \cV$ are considered adjacent if and only if $\left| v - w \right| = 1$. Additionally, let $T$ denote a two-dimensional coordinate transformation, which is defined as follows:
\begin{equation*}
    T\left( x_1,x_2 \right) = \left( x_1 + \frac{x_2}{2},\frac{x_2\sqrt{3}}{2} \right)\,.
\end{equation*}

\begin{figure}[H]
    \centering
    \includegraphics[width=\textwidth,height=\textheight,keepaspectratio]{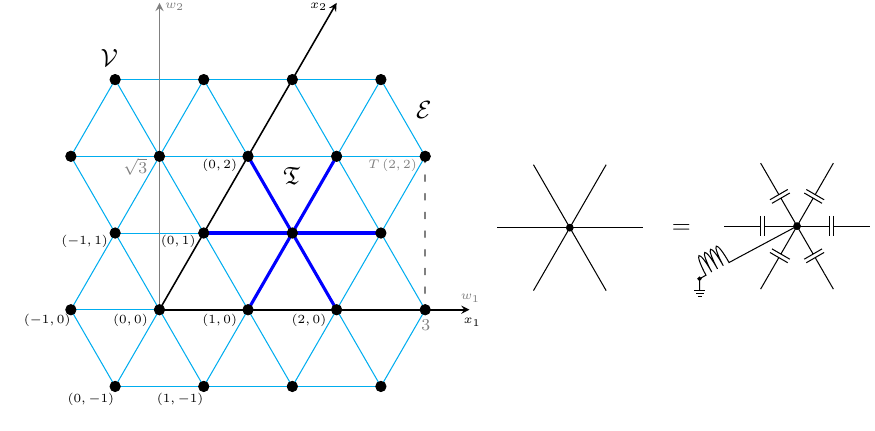}
    \caption{A fragment of a two-dimensional infinite triangular lattice, demonstrating its conceptual sense.}
    \label{fig:frag_triang_lattice}
\end{figure}
In \hyperref[fig:frag_triang_lattice]{Figure \ref*{fig:frag_triang_lattice}}, a fragment of a two-dimensional infinite triangular lattice $\mathfrak{T}$ is illustrated. The connection between $x = \left( {x}_{1},{x}_{2} \right) \in \bZ^2$ and the Euclidean coordinates of the vertexes (black dots) is established through the mapping $\left( {x}_{1},{x}_{2} \right) \mapsto \left( {x}_{1} + {x}_{2} / 2,\sqrt{3}{x}_{2} / 2 \right)$. For illustrative purposes, \hyperref[fig:frag_triang_lattice]{Figure \ref*{fig:frag_triang_lattice}} additionally depicts the lattice vertex $T\left( 2,2 \right)$, with Cartesian coordinates $\left( 3,\sqrt{3} \right)$. It is derived from the point $\left( 2,2 \right) \in \bZ^2$ through the coordinate transformation $T$. The nearest neighbour interactions on the triangular lattice are depicted with thick blue lines. Conceptually, the triangular lattice can be understood as a two-dimensional left-handed metamaterial consisting of inductor-capacitor elements. Specifically, it can be likened to a periodically loaded host transmission line with series capacitors and shunt inductors.

We define the $6$-neighbourhood ${F}_{w}^{0}$ as the set of points $v \in \cV$ such that the distance between $w$ and $v$ is equal to $1$, where $w \in \cV$. Additionally, we define the neighbourhood ${F}_{w}$ as the union of ${F}_{w}^{0}$ and the singleton set containing $w$. Further, a region $R \subset \cV$ can be identified as a set consisting of points in $\cV$, where there exist disjoint nonempty subsets $\Rint$ and $\partial R$ of $R$ satisfying the following condition:
\begin{enumerate}[label={(\alph*)}]
    \item $R = \Rint \cup \partial R$,
    \item if $w\in \Rint$ then ${F}_{w} \subset R$,
    \item if $w \in \partial R$ then there is at least one point $v \in {F}_{w}^{0}$ such that $v \in \Rint$.
\end{enumerate}

Since subsets $\Rint$ and $\partial R$ are not uniquely defined by $R$, we assume that for a given region $R$ in $\cV$, $\Rint$ and $\partial R$ are given and fixed. We define a point $w$ as an interior or boundary point of $R$ if $w \in \Rint$ or $w \in \partial R$, respectively. Additionally, we define a region $R \subset \cV$ as connected if for any $v, \tilde{v} \in R$, there exists a sequence ${w}^{\left( 1 \right)}, \dots, {w}^{\left( n \right)} \in R$ such that ${w}^{\left( 1 \right)} = v$, ${w}^{\left( n \right)} = \tilde{v}$, and for all $0 \leq i \leq n-1$, $\left| {w}^{\left( i \right)} - {w}^{\left( i + 1 \right)} \right| = 1$. By definition, a region $R$ with a single interior point $w$ is connected and coincides with ${F}_{w}$.

Let $\Gamma_j$, where $j = 1,2,\dots,n$ and $n \in \mathbb{N}$, denote a finite row of lattice sites. Consider a region $\mathring{\cW}$ defined as the set difference $\mathring{\cW}=\cV \setminus \partial \cW$, where $\partial \cW$ is the union of the sets $\Gamma_j$, for $j = 1,2,\dots,n$. We assume that the lattice sites $\Gamma_j$ are arranged in a manner such that $\mathring{\cW}$ is a connected region and satisfies the cone condition, as defined in reference \cite{KaPes2023}. Lastly, we emphasize that $\cW$ is the union of $\mathring{\cW}$ and $\partial \cW$, and that $\mathcal{V}$ and $\cW$ are equivalent as sets.

Given the problem and assumptions stated in the introduction, we consider a system in which each node $w \in \mathring{\cW}$ is connected to a common ground plane via an inductor, and each node $w \in \mathring{\cW}$ is connected to its six nearest neighbours via a capacitor. We assume that all inductances are equal to a positive constant $L$, and all capacitances are equal to a positive constant $C$. Under these conditions, Kirchhoff's laws of voltage and current imply a second-order differential equation for the voltage $U\left( w,t \right)$ across the inductor at node $w$, as follows:
\begin{equation}\label{eq:inductor-capacitor_probl}
    {L}{C}\frac{{\mathrm{d}}^{2}}{\mathrm{d}{t}^{2}}\left( \Delta_{d}U\left( w,t \right) \right) = U\left( w,t \right)\,.
\end{equation}
Here, the discrete Laplacian operator $\Delta_d$ is represented by a $7$-point stencil and can be expressed as:
\begin{equation*}
    \Delta_{d} U\left( w,t \right) = \sum_{v \in {F}_{w}^{0}}U\left( v,t \right) - 6U\left( w,t \right)\,.
\end{equation*}

It should be noted that equation \eqref{eq:inductor-capacitor_probl} is valid for all $w \in \mathring{\cW}$, and the boundary of the region $\cW$, denoted by $\partial \cW$, is subject to a time-dependent boundary condition specified by a function $f\left( v \right)$, where $v$ denotes a node on the boundary. The boundary condition is given by:
\begin{equation}\label{eq:bound_cond_inductor-capacitor}
    U\left( v,t \right) = f\left( v \right){e}^{-\imath \omega t}\,.
\end{equation}
It is important to note that $\imath$ is used to denote the imaginary unit, and $f: \partial \cW \to \bC$ is a given function. Under the assumption that at time $t = 0$, $U\left( w,t \right)$ and all its derivatives are zero for all $w \in \mathring{\cW}$, a wave propagates into the lattice as $t$ increases due to the time-dependent boundary condition, eventually leading the system to a steady state. In this steady state, the solution takes the form $U\left( w,t \right) = u\left( w \right){e}^{-\imath \omega t}$, where $u\left( w \right)$ is a complex-valued function that is constant in time. By substituting this expression into equations \eqref{eq:inductor-capacitor_probl} and \eqref{eq:bound_cond_inductor-capacitor}, which correspond to the discrete Helmholtz equation in $\cW$, we arrive at the following problem:
\begin{subequations}
    \begin{alignat}{2}
        \left( \Delta_{d} + k^2 \right)u\left( w \right) &= 0\,, &\quad \text{in}& ~ \mathring{\cW}\,,\label{eq:H1} \\
        u\left( v \right) &= f\left( v \right)\,, &\quad \text{on}& ~ \partial \cW\,.\label{eq:H2}
    \end{alignat}
\end{subequations}
The problem \eqref{eq:H1}-\eqref{eq:H2} involves the Laplacian operator $\Delta_{d}$ and a parameter $k$ that is related to the frequency $\omega$ of the wave through the equation ${k}^{2} = {\left( {\omega}^{2}{L}{C} \right)}^{-1}$, where $L$ and $C$ are the inductance and capacitance of the lattice, respectively.

To enhance efficiency, we have opted to work within a simplified coordinate system that is denoted by ${\bZ}^{2}$, as employed in \eqref{eq:Vmain}. Specifically, we represent each node $w$ within the set $\cW$ by a pair of integers, $\left( x_1, x_2 \right)$. Within this system, every node is associated with six neighbours: $\left( x_1 - 1, x_2 \right)$, $\left( x_1 + 1, x_2 \right)$, $\left( x_1, x_2 - 1 \right)$, $\left( x_1, x_2 + 1 \right)$, $\left( x_1 + 1, x_2 - 1 \right)$, and $\left( x_1 - 1, x_2 + 1 \right)$. Consequently, we assume all definitions that have been previously introduced with respect to content using integer coordinates, without explicit mention.

Let us reformulate equation \eqref{eq:inductor-capacitor_probl} within the context of the selected coordinate system. To be precise, we set $U\left( x,t \right) = U\left( w,t \right)$, where $x = \left( x_1,x_2 \right) \in {\bZ}^{2}$ and $w = T\left( x_1,x_2 \right) \in \mathcal{W}$. With these definitions, we obtain:
\begin{equation}\label{eq:inductor-capacitor_probl_x}
    {L}{C}\frac{{\mathrm{d}}^{2}}{\mathrm{d}{t}^{2}}\left( \Delta_{d}U\left( x,t \right) \right) = U\left( x,t \right)\,.
\end{equation}
If we reformulate the boundary condition \eqref{eq:bound_cond_inductor-capacitor} in the coordinate system ${\bZ}^{2}$ and follow the same reasoning as above, we obtain the following boundary value problem of Dirichlet-type for the time-harmonic discrete waves in $\Omega = \mathring{\Omega} \cup \partial \Omega \subset {\bZ}^{2}$
\begin{subequations}
    \begin{alignat}{2}
        \left( \Delta_{d} + k^2 \right)u\left( x \right) &= 0\,, &\quad \text{in}& ~ \mathring{\Omega}\,,\label{eq:Helmholtz} \\
        u\left( y \right) &= f\left( y \right)\,, &\quad \text{on}&~ \partial \Omega\,,\label{eq:H3}
    \end{alignat}
\end{subequations}
where the discrete Laplacian operator is expressed as (its action is represented by a stencil as illustrated in \hyperref[fig:7-point_stencil_of_Laplacian]{Figure \ref*{fig:7-point_stencil_of_Laplacian}})
\begin{equation*}
    \begin{aligned}
        {\Delta}_{d} u\left( x \right) =& u\left( x + e_1 \right) + u\left( x - e_1 \right) + u\left( x + e_2 \right) + u\left( x - e_2 \right) \\
        &+ u\left( x + e_1 - e_2 \right) + u\left( x - e_1 + e_2 \right) - 6u\left( x \right)\,,
    \end{aligned}
\end{equation*}
\begin{figure}[H]
    \centering
    \includegraphics[width=0.30\textwidth]{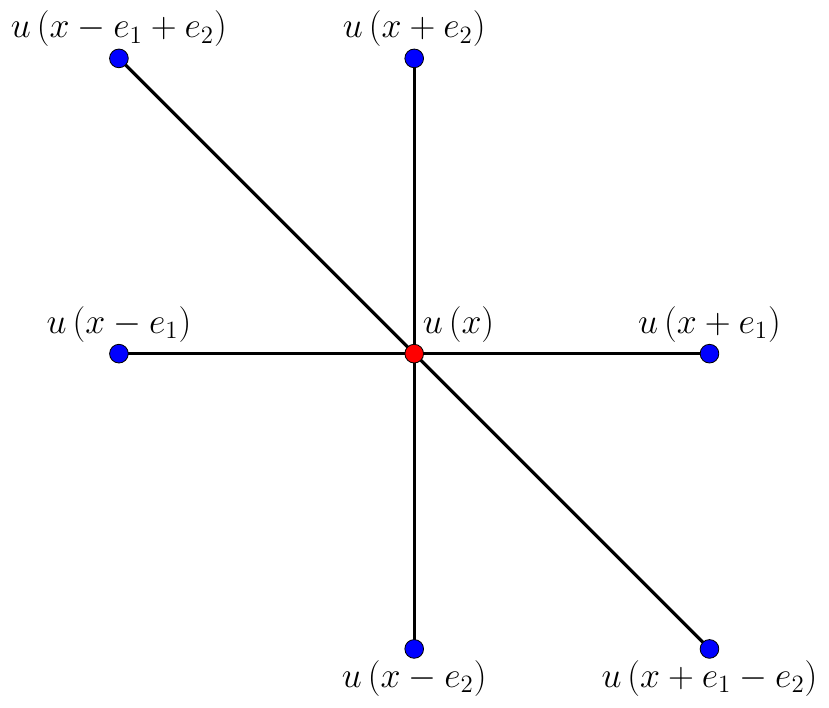}
    \caption{$7$-point stencil of the discrete Laplacian operator ${\Delta}_{d}$ in the coordinate system ${\bZ}^{2}$.}
    \label{fig:7-point_stencil_of_Laplacian}
\end{figure}
\noindent $e_1 = \left( 1,0 \right)$ and $e_2 = \left( 0,1 \right)$ denote the standard basis vectors of ${\bZ}^{2}$, representing the directions of the positive $x_1$ and $x_2$ axes, respectively. $\mathring{\Omega}$ denotes the interior of the domain, and $\partial \Omega$ denotes the boundary of the domain. The function $f: \partial \Omega \rightarrow {\bC}$ is a given boundary condition. Note that the coordinate system ${\bZ}^{2}$ is used to represent the nodes of the discrete system, and all previously introduced definitions apply in this context.

It is established that equation \eqref{eq:inductor-capacitor_probl} permits plane wave solutions $U\left( x \right) = {A}{e}^{-\imath\left( {\xi}_{1}{x}_{1} + {\xi}_{2}{x}_{2} + {\omega}{t} \right)}$, where $A \in \bC$ denotes a constant, provided that the ensuing dispersion relation holds:
\begin{equation*}
    {\omega}^{2} = \dfrac{1}{{4}{L}{C}\left({\sin}^{2}\frac{\xi_1}{2} + {\sin}^{2}\frac{\xi_2}{2} + {\sin}^{2}\frac{\xi_1 - \xi_2}{2} \right)}\,.
\end{equation*}
Assuming the fulfilment of the dispersion relation for $\left( \xi_1, \xi_2 \right) \in {\left[ -\pi, \pi \right]}^{2}$, identified as the first Brillouin zone (see \cite{Brillouin1953SecondEd}). Consequently, the following holds:
\begin{equation*}
    {k}^{2} = {4}\left({\sin}^{2}\frac{\xi_1}{2} + {\sin}^{2}\frac{\xi_2}{2} + {\sin}^{2}\frac{\xi_1 - \xi_2}{2} \right) \in \left[ 0,9 \right]\,.
\end{equation*}
Undoubtedly, other values of $k$ are also subject to investigation, but these cases are relatively straightforward and are not addressed within the scope of this discussion.

As a reminder, the spectrum of the negative discrete Laplacian is known to be absolutely continuous within the interval $\left[ 0,9 \right]$. However, there exists a set of exceptional points $\left\{ 0,8,9 \right\}$ within this interval where the limiting absorption principle fails, as detailed in \cite{ShabanVainberg2001}. Therefore, it is assumed that $k \in \left( 0,3 \right) \setminus \left\{ 2\sqrt{2} \right\}$.

The objective of this study is to investigate the existence and uniqueness of the function $u:\Omega \to \mathbb{C}$ that satisfies the discrete Helmholtz equation \eqref{eq:Helmholtz} with $k \in \left( 0,2\sqrt{2} \right)$ and the boundary condition \eqref{eq:H3}. This problem shall be denoted as {\Prp} in the following discussion.

\section{Green's representation formula}\label{sec:Green's_represent_form}
This section primarily recollects the results from \cite{KaPes2023,KapanadzePesetskaya2023,Kapanadze2021}. Let $R$ be a region in the two-dimensional integer lattice ${\bZ}^2$. For $j = 1,\dots,6$, let ${\left( \partial R \right)}_{j}$ denote the set of all boundary points $y \in \partial R$ such that $y - {e}_{j} \in \Rint$, where $\Rint$ denotes the interior of $R$. We refer to ${\left( \partial R \right)}_{j}$ as the $j$-th side of the boundary of $R$. Note that $\partial R$ is the union of its six sides, {\ie}, $\partial R = {\cup}_{j = 1}^{6} {\left( \partial R \right)}_{j}$. Although a boundary point $y$ can belong to all six sides of $R$, it will always be evident from our subsequent arguments which side needs to be considered. Subject to this clarification, we define the discrete derivative in the outward normal direction ${e}_{j}$, $j = 1,\dots,6$ as follows:
\begin{equation}\label{eq:T_operator}
    \mathcal{T} u\left( y \right) = u\left( y \right) - u\left( y - {e}_{j} \right)\,,\quad y \in (\partial R)_j\,,
\end{equation}
where ${e}_{3} = {e}_{1} - {e}_{2}$, ${e}_{4} = -{e}_{1}$, ${e}_{5} = -{e}_{2}$ and ${e}_{6} = -{e}_{3}$.

Let ${H}_{0} = \left\{ \left( 0,0 \right) \right\}$, and for each $N \in \bN$, define ${H}_{N}$ by the recurrence formula
\begin{equation*}
    {H}_{N} \coloneqq \bigcup_{x \in {H}_{N-1}}{F}_{x}\,,
\end{equation*}
here $\mathring{H}_{N} \coloneqq {H}_{N-1}$ and ${\left( \partial H \right)}_{N} \coloneqq {H}_{N} \setminus \mathring{H}_{N}$.

Let $R$ be a finite region, and let $R = \bigcup_{x \in \Rint}{F}_{x}$ be a representation of $R$. Then we have discrete analogues of both Green's first and second identities, given by:
\begin{equation}\label{eq:GRfirst}
    \sum_{x \in \Rint}\left( {\nabla}_{d}^{+} u\left( x \right) \cdot {\nabla}_{d}^{+} v\left( x \right) + {\nabla}_{d}^{-} u\left( x \right) \cdot {\nabla}_{d}^{-} v\left( x \right) + u\left( x \right){\Delta}_{d} v\left( x \right) \right) = \sum_{y \in \partial R} u\left( y \right) \mathcal{T} v\left( y \right)\,,
\end{equation}
and
\begin{equation}\label{eq:GRsecond}
    \sum_{x \in \Rint}\left( u\left( x \right){\Delta}_{d} v\left( x \right) - v\left( x \right){\Delta}_{d} u\left( x \right) \right) = \sum_{y \in \partial R} \left( u\left( y \right)\mathcal{T}v\left( y \right) - v\left( y \right)\mathcal{T}u\left( y \right) \right)\,,
\end{equation}
respectively. Here, ${\nabla}_{d}^{+}$ and ${\nabla}_{d}^{-}$ are defined as follows:
\begin{equation*}
    {\nabla}_{d}^{+} \coloneqq 
    \begin{pmatrix}
        u\left( x + {e}_{1} \right) - u\left( x \right) \\
        u\left( x + {e}_{2} \right) - u\left( x \right) \\
        u\left( x + {e}_{3} \right) - u\left( x \right)
    \end{pmatrix}\,,
\end{equation*}
and
\begin{equation*}
    {\nabla}_{d}^{-} \coloneqq 
    \begin{pmatrix}
        u\left( x - {e}_{1} \right) - u\left( x \right) \\
        u\left( x - {e}_{2} \right) - u\left( x \right) \\
        u\left( x - {e}_{3} \right) - u\left( x \right)
    \end{pmatrix}\,.
\end{equation*}

Let $\cG\left( x;y \right)$ denote Green's function for the discrete Helmholtz equation \eqref{eq:Helmholtz} centred at $y$ and evaluated at $x$. Then, the function $\cG\left( x;y \right)$ satisfies the following equation:
\begin{equation}\label{eq:greens}
    \left( \Delta_{d} + k^2 \right)\cG\left( x;y \right) = \delta_{x,y}\,,
\end{equation}
where $\delta_{x,y}$ represents the Kronecker delta. To simplify the notation, we denote $\cG\left( x;0 \right)$ as $\cG\left( x \right)$. It is important to note that $\cG\left( x;y \right)$ can be represented as $\cG\left( x - y \right)$.

The lattice Green's function $\cG$ is a well-known concept and has been extensively studied in the literature, see: \cite{Economou2006,KatsuraInawashiro1971,Martin2006,Zemla1995}. By employing the discrete Fourier and the inverse Fourier transforms, we can obtain the following result:
\begin{equation}\label{eq:gmn}
    \cG\left( x \right) = \frac{1}{4\pi^2}\int\limits_{-\pi}^{\pi}\int\limits_{-\pi}^{\pi}\frac{e^{\imath\left( x\cdot \xi \right)}}{\sigma\left( \xi;k^2 \right)}\,\mathrm{d}\xi\,,\quad \xi = \left( \xi_1,\xi_2 \right)\,,
\end{equation}
where
\begin{equation*}
    \begin{aligned}
        \sigma\left( \xi;k^2 \right) &= e^{\imath\xi_1} + e^{-\imath\xi_1} + e^{\imath\xi_2} + e^{-\imath\xi_2} + e^{\imath\xi_1}e^{-\imath\xi_2} + e^{-\imath\xi_1}e^{\imath\xi_2} - 6 + k^2 \\
        &= k^2 - 6 + 2\cos\xi_1 + 2\cos\xi_2 + 2\cos\left( \xi_1 - \xi_2 \right) \\
        &= k^2 - 4\left({\sin}^{2}\frac{\xi_1}{2} + {\sin}^{2}\frac{\xi_2}{2} + {\sin}^{2}\frac{\xi_1 - \xi_2}{2} \right)\,.
    \end{aligned}
\end{equation*}

The lattice Green's function $\cG$ is a well-established concept when $k^2$ belongs to the set $\bC \setminus \left[ 0,9 \right]$ ({\cf}, {\eg}, \cite{Horiguchi1972}). Notably, if $k^2 \in \bC \setminus \left[ 0, 9 \right]$, then $\sigma \neq 0$ and $\cG$ in \eqref{eq:gmn} possesses a well-defined value. In such cases, the function $\cG\left( x \right)$ exhibits an exponential decay as $\left| x \right| \to \infty$.

When $k \in \left( 0, 2\sqrt{2} \right)$, we define the lattice Green's function as the point-wise limit of the following expression:
\begin{equation}\label{eq:pointwise}
    \left( R_{k^2 + \imath\varepsilon}\delta_{x,0} \right)\left( x \right) \coloneqq \frac{1}{4\pi^2}\int\limits_{-\pi}^{\pi}\int\limits_{-\pi}^{\pi}\frac{e^{\imath x\cdot\xi }\,\mathrm{d}\xi_1\mathrm{d}\xi_2}{\sigma\left( \xi;k^2 + \imath\varepsilon \right)}\,,
\end{equation}
as $k^2 + \imath\varepsilon$ approaches $k^2 + \imath 0$, we denote the resulting limit by $\cG\left( x \right)$, {\ie}, $\cG\left( x \right) = \left( R_{k^2 + \imath 0}\delta_{x,0} \right)\left( x \right)$, {\cf}, \cite{Kapanadze2021}. It is noteworthy that $\cG\left( x \right)$ satisfies the equation \eqref{eq:greens} and the following equalities for all $x = \left( x_1,x_2 \right) \in \mathbb{Z}^2$ (see \cite{KapanadzePesetskaya2023}):
\begin{equation}\label{eq:property_Green's_funct.}
    \cG\left( x_1,x_2 \right) = \cG\left( x_2,x_1 \right) = \cG\left( -x_1,-x_2 \right) = \cG\left( x_1 + x_2,-x_2 \right)\,.
\end{equation}

Let us reconsider the expression for the lattice Green's function \eqref{eq:gmn} with a different formulation. By employing the sum-to-product identities and double-angle formula for cosine within $\sigma\left( \xi;k^2 \right)$, we derive:
\begin{equation*}
    \sigma\left( \xi;k^2 \right) = {k}^{2} - 8 + 4 \cos\frac{\xi_{1} - \xi_{2}}{2}\left( \cos\frac{\xi_{1} + \xi_{2}}{2} + \cos\frac{\xi_{1} - \xi_{2}}{2} \right)\,.
\end{equation*}
If we introduce a change of variables as follows:
\begin{equation*}
    \eta_{1} = \frac{\xi_{1} + \xi_{2}}{2} \in \left[ -\pi,\pi \right]\,,\quad \eta_{2} = \frac{\xi_{1} - \xi_{2}}{2} \in \left[ -\pi,\pi \right]\,,
\end{equation*}
or equivalently,
\begin{equation*}
    \xi_{1} = \eta_{1} + \eta_{2}\,,\quad \xi_{2} = \eta_{1} - \eta_{2}\,,
\end{equation*}
we can define the function
\begin{equation*}
    \widetilde{\sigma}\left( \eta;k^2 \right) = {k}^{2} - 8 + 4\cos\eta_{2}\left( \cos\eta_{1} + \cos\eta_{2} \right)\,,\quad \eta = \left( {\eta}_{1},{\eta}_{2} \right)\,.
\end{equation*}
In this change of variables, the integration domain $\left[ -\pi,\pi \right] \times \left[ -\pi,\pi \right]$ is rotated by $\pi/2$ around the origin and compressed by a factor of $1/\sqrt{2}$. Consequently, the lattice Green's function, denoted by \eqref{eq:gmn}, can be expressed in terms of new variables ${\eta}_{1}$ and ${\eta}_{2}$ as follows
\begin{equation*}
    \cG\left( {x}_{1},{x}_{2} \right) = \frac{1}{4\pi^2}\int\limits_{-\pi}^{\pi}\int\limits_{-\pi}^{\pi}\frac{e^{\imath\left( x_{1}\left( \eta_{1} + \eta_{2} \right) +x_{2}\left( \eta_{1} - \eta_{2} \right) \right)}}{\widetilde{\sigma}\left( {\eta}_{1},{\eta}_{2};k^2 \right)}\,\mathrm{d}\eta_{2}\mathrm{d}\eta_{1}\,.
\end{equation*}
\begin{theorem}[see \cite{Kapanadze2021}]\label{theo:repr}
    Consider a finite region $R$ and a function $u: R \to \bC$. For any point, $x \in \Rint$, a discrete Green's representation formula can be established
    \begin{equation*}
        u\left( x \right) = \sum_{y \in \partial R}\bigl( u\left( y \right)\mathcal{T}\cG\left( x - y \bigr) - \cG\left( x-y \right)\mathcal{T}u\left( y \right) \right) + \sum_{y \in \Rint}\cG\left( x - y \right)\left( \Delta_{d} + k^2 \right) u\left( y \right)\,.
    \end{equation*}
    Furthermore, if $u$ satisfies the discrete Helmholtz equation
    \begin{equation*}
        \left( \Delta_d + k^2 \right)u\left( x \right) = 0\,,\quad \text{in}~ \Rint,
    \end{equation*}
    then the following representation formula holds:
    \begin{equation}\label{eq:repr}
        u\left( x \right) = \sum_{y \in \partial R}\bigl(u\left( y \right)\mathcal{T}\cG\left( x - y \right) - \cG\left( x - y \right)\mathcal{T}u\left( y \right) \bigr)\,.
    \end{equation}
    Recall that, here $\partial R$ refers to the boundary of a finite region $R$.
\end{theorem}
Moreover, we must apply the concept of a radiation condition to the discrete Helmholtz operators. It should be emphasized that in the case where $k^2 \in \left( 0,8 \right)$, an extra condition at infinity is necessary ({\cf}, \cite{ShabanVainberg2001}). Specifically, a function $u: \Omega \to \bC$ satisfies the radiation condition at infinity if (see \cite{Kapanadze2021})
\begin{equation}\label{eq:radcond}
    \left\{
    \begin{aligned}
        u\left( x \right) &= \mathcal{O}\left( {\left| x \right|}^{ -\frac{1}{2} } \right)\,,\\
        u\left( x + {e}_{j} \right) &= e^{\imath {\xi}_{j}^{*}\left( \alpha,k \right)} u\left( x \right) + \mathcal{O}\left( {\left| x \right|}^{ -\frac{3}{2} } \right)\,,\quad j = 1,2\,.
    \end{aligned}
    \right.
\end{equation}
Here, the term remaining after the decay is uniform in all directions $x\,/\left| x \right|$, where $x$ is characterized by ${x}_{1} = \left| x \right|\cos\alpha$ and ${x}_{2} = \left| x \right|\sin\alpha$, $0 \leq \alpha < 2\pi$. Here, ${\xi}_{j}^{*}\left( \alpha,k \right)$ is the $j$-th coordinate of the point ${\xi}^{*}\left( \alpha,k \right)$, where ${\xi}^{*}\left( \alpha,k \right) = {\xi}^{*} = \left( {\xi}_{1}^{*},{\xi}_{2}^{*} \right)$ is the unique solution to the following system of equations:
\begin{align*}
    2\zeta\left( \sin\xi_1 + \sin\left( \xi_1 - \xi_2 \right) \right) &= \cos\alpha\,, \\
    2\zeta\left( \sin\xi_2 - \sin\left( \xi_1 - \xi_2 \right) \right) &= \sin\alpha\,, \\
    k^2 - 6 + 2\cos\xi_1 + 2\cos\xi_2 + 2\cos\left( \xi_1 - \xi_2 \right) &= 0\,,
\end{align*}
where $\zeta$ is a positive constant ({\cf}, \cite{Kapanadze2021}).
\begin{definition}\label{def:radiating}
    Let $k$ be an element of the open interval $\left( 0,2\sqrt{2} \right)$. Consider a solution $u$ to the discrete Helmholtz equation \eqref{eq:Helmholtz}, which is a second-order difference equation used to model wave phenomena in a triangular lattice. We say that $u$ is a radiating solution if it satisfies the radiation condition \eqref{eq:radcond}, which ensures that the solution decays at infinity and represents a physically realistic wave field.
\end{definition}

Observe that the second condition presented in \eqref{eq:radcond} can be formulated in a subsequent manner:
\begin{equation*}
    u\left( x \right) - u\left( x + e_j \right) = \left( 1 - e^{\imath {\xi}_{j}^{*}\left( \alpha,k \right)} \right) u\left( x \right) + \mathcal{O}\left( {\left| x \right|}^{ -\frac{3}{2} } \right)\,,\quad \left| x \right| \to \infty\,,
\end{equation*}
and
\begin{equation*}
    u\left( x + e_j \right) - u\left( x \right) = \left( 1 - e^{-\imath {\xi}_{j}^{*}\left( \alpha,k \right)} \right) u\left( x + e_j \right) + \mathcal{O}\left( {\left| x \right|}^{ -\frac{3}{2} } \right)\,,\quad \left| x \right| \to \infty\,.
\end{equation*}
\begin{lemma}[For the detailed proof see \cite{KaPes2023}]\label{lemma:zeta}
    Suppose that $k^2$ belongs to the interval $\left( 0,8 \right)$, and consider a function $u$ that satisfies the radiation condition at infinity, as expressed by equation \eqref{eq:radcond}. For any point $y$ on the boundary of $H_N$, we have the following asymptotic expansion as $N$ tends to infinity:
    \begin{equation*}
        \mathcal{T}u\left( y \right) = \zeta\left( y,k \right) u\left( y \right) + \mathcal{O}\left({\left| y \right|}^{-\frac{3}{2}} \right)\,,\quad N \to \infty\,,
    \end{equation*}
    where $\zeta\left( y,k \right)$ is a complex-valued function satisfying $\mathrm{\Im m}\,\zeta\left( y,k \right) > 0$.
\end{lemma}

Consider a fixed point $x \in \bZ^2$, and let $y$ be any point on the boundary of $H_{N}$. The radiation conditions expressed by equation \eqref{eq:radcond} imply that
\begin{equation*}
    \sum_{y \in \partial H_{N}}\left( u\left( y \right)\mathcal{T}\cG\left( x - y \right) - \cG\left( x - y \right)\mathcal{T}u\left( y \right) \right) \to 0\,,\quad N \to \infty\,.
\end{equation*}
Undoubtedly, as an example, it can be observed that as $\left| y \right| \to \infty$, $\alpha\left( y - x \right)$ converges to $\alpha = \alpha\left( y \right)$ and $\zeta\left( y - x,k \right)$ converges to $\zeta\left( y,k \right)$, implying that for $N$ sufficiently large, the ensuing result can be obtained (see \cite{KaPes2023}):
\begin{align*}
    u\left( y \right)\mathcal{T}\cG\left( x - y \right) &- \cG\left( x - y \right)\mathcal{T}u\left( y \right) \\
    &= u\left( y \right) \cdot \mathcal{O}\left( {N}^{-\frac{3}{2}} \right) + \cG\left( x - y \right) \cdot \mathcal{O}\left( {N}^{-\frac{3}{2}} \right) = \mathcal{O}\left( {N}^{-2} \right)\,.
\end{align*}
Moreover, by applying \hyperref[theo:repr]{Theorem \ref*{theo:repr}} to $\Omega\cap {H}_{N}$, where $N \in \bN$ is sufficiently large, and subsequently taking the limit as $N \to \infty$, the ensuing Green's formula for a radiating solution $u$ of the discrete Helmholtz equation \eqref{eq:Helmholtz} can be derived:
\begin{equation}\label{eq:repre}
    u\left( x \right) = \sum_{y \in \partial \Omega}\left( u\left( y \right)\mathcal{T}\cG\left( x - y \right) - \cG\left( x - y \right)\mathcal{T}u\left( y \right) \right)\,.
\end{equation}

Based on the representation formula \eqref{eq:repre} and using the results established in \cite{Kapanadze2021}, it can be deduced that any radiating solution $u$ of the discrete Helmholtz equation \eqref{eq:Helmholtz} possesses the ensuing asymptotic expansion:
\begin{equation}\label{eq:far}
    u\left( x \right) = -\frac{{e}^{\imath \mu\left( \alpha,k \right)\left| x \right|}}{{\left| x \right|}^{\frac{1}{2}}}\left\{ {u}_{\infty}\left( \hat{x} \right) + \mathcal{O}\left( \frac{1}{\left| x \right|} \right) \right\}\,,\quad \left| x \right| \to \infty\,.
\end{equation}
Furthermore, let us define $\mu\left( \alpha,k \right) \coloneqq {\xi}^{*}\left( \alpha,k \right) \cdot \hat{x}$, where $\hat{x} \coloneqq x\,/\left| x \right|$ and ${\xi}^{*}\left( \alpha,k \right) = \left( {\xi}_{1}^{*}\left( \alpha,k \right),{\xi}_{2}^{*}\left( \alpha,k \right) \right)$. The far-field pattern of $u$ denoted as ${u}_{\infty}\left( \hat{x} \right)$, can be expressed through the employing of the formula $\left( 12 \right)$ introduced in \cite{Kapanadze2021}.

Now, based on the previous results, we are ready to state the following proposition:
\begin{theorem}\label{th:uni}
    The {\Prp} possesses, at most, a single solution that exhibits radiation properties.
\end{theorem}
\begin{proof}
    Demonstrating that the corresponding homogeneous problem has only the trivial solution is a sufficient condition to establish the result.

    Let $\Omega \cap {H}_{N}$ denote the intersection of the domain $\Omega$ with the finite rectangular grid ${H}_{N}$. In references \cite{Kapanadze2021,KaPes2023}, the discrete version of Green's first identity, applied in $\Omega \cap {H}_{N}$, is expressed as follows:
    \begin{equation}\label{eq:thm_GRfirst}
        \sum_{x \in \Omega \cap {H}_{N}}\left( {\nabla}_{d}^{+} u\left( x \right) \cdot {\nabla}_{d}^{+} v\left( x \right) + {\nabla}_{d}^{-} u\left( x \right) \cdot {\nabla}_{d}^{+} v\left( x \right) + u\left( x \right){\Delta}_{d} v\left( x \right) \right) = \sum_{y \in \partial {H}_{N}} u\left( y \right)\mathcal{T}v\left( y \right)\,.
    \end{equation}
    Setting $v \coloneqq \overline{u}$ in \eqref{eq:thm_GRfirst}, we obtain the following result:
    \begin{equation}\label{eq:SN}
        \sum_{x \in \Omega \cap {H}_{N}}\left({\left| {\nabla}_{d}^{+} u\left( x \right) \right|}^{2} + {\left| {\nabla}_{d}^{-} u\left( x \right) \right|}^{2} - {k}^{2}{\left| u(x) \right|}^{2} \right) = \sum_{y \in \partial {H}_{N}}u\left( y \right)\mathcal{T}\overline{u}\left( y \right)\,.
    \end{equation}
    By applying \hyperref[lemma:zeta]{Lemma \ref*{lemma:zeta}}, we can reformulate equality \eqref{eq:SN} in the following manner
    \begin{equation*}
        \sum_{x \in \Omega \cap {H}_{N}}\left({\left| {\nabla}_{d}^{+} u\left( x \right) \right|}^{2} + {\left| {\nabla}_{d}^{-} u\left( x \right) \right|}^{2} - {k}^{2}{\left| u(x) \right|}^{2} \right) = \sum_{y \in \partial {H}_{N}} \overline{\zeta}\left( y,k \right){\left| u\left( y \right) \right|}^{2} + \mathcal{O}\left( {N}^{-1} \right)\,.
    \end{equation*}
    By taking the imaginary part of the final identity and then letting $N$ tend towards infinity, we obtain the limit
    \begin{equation*}
        \sum_{y \in \partial {H}_{N}} {\left| u\left( y \right) \right|}^{2} \to 0\,,\quad N \to \infty\,.
    \end{equation*}
    Furthermore, employing a Rellich-type theorem \cite{IsozakiMorioka2014,Kapanadze2018,AndoIsozakiMorioka2016}, we can assert that for sufficiently large $N$, the function $u$ vanishes outside of ${H}_{N}$. Since $\Omega$ satisfies the cone condition, we can apply the unique continuation property \cite[Theorem 5.7]{AndoIsozakiMorioka2016}, which implies that $u$ is identically zero throughout $\Omega$.
\end{proof}

\section{Difference potentials and the existence of a solution}\label{sec:diff_potent}
For any given function $\varphi: \partial R \to \bC$, the difference single-layer and double-layer potentials are defined as:
\begin{equation}\label{eq:V}
    V\varphi\left( x \right) = \sum_{y \in \partial R}\cG\left( x - y \right)\varphi\left( y \right)\,,\quad \text{for all}\ x \in \bZ^2\,,
\end{equation}
and
\begin{equation}\label{eq:W}
    W\varphi\left( x \right) = \sum_{y \in \partial R}\bigr( \mathcal{T}\cG\left( x - y \right) + \delta_{x,y} \bigl)\varphi\left( y \right)\,,\quad  \text{for all}\ x \in \bZ^2\,,
\end{equation}
respectively. Since $\delta_{x,y} = 0$ for every $x \in \Rint$ and $y \in \partial R$, expression \eqref{eq:repre} can be written as:
\begin{equation*}
    u\left( x \right) = Wu\left( x \right) - V\left( \mathcal{T}u \right)\left( x \right)\,,\quad x \in \mathring{\Omega}\,.
\end{equation*}
The role of the summand $\delta_{x,y}$ is clarified by the following result.
\begin{lemma}\label{lem:VW}
    For every $x \in \Rint$, we have:
    \begin{equation*}
        \left( {\Delta}_{d} + k^2 \right)V\varphi\left( x \right) = 0\,,\quad \text{and}\quad \left( {\Delta}_{d} + k^2 \right)W\varphi\left( x \right) = 0\,.
    \end{equation*}
    This indicates that both the difference single-layer potential $V\varphi\left( x \right)$ and the double-layer potential $W\varphi\left( x \right)$ satisfy the discrete Helmholtz equation with wavenumber $k^2$.
\end{lemma}
\begin{proof}
    For the difference single-layer potential \eqref{eq:V}, we can obtain the following result:
    \begin{equation*}
        \left( {\Delta}_{d} + k^2 \right)V\varphi\left( x \right) = \sum_{y \in \partial R}\left[ \left( {\Delta}_{d} + k^2 \right)\cG\left( x - y \right) \right]\varphi\left( y \right) = \sum_{y \in \partial R}\delta_{x,y}\varphi\left( y \right) = 0\,,\quad \forall x \notin \partial R\,.
    \end{equation*}
    Likewise, if $x \in \Rint$ and ${F}_{x}\cap\partial R = \varnothing$, we can establish the outcome for the difference double-layer potential \eqref{eq:W}. Specifically, we have $\left( {\Delta}_{d} + k^2 \right)\mathcal{T}\cG\left( x - y \right) = 0$ and $\left( {\Delta}_{d} + k^2 \right)\delta_{x,y} = 0$. Thus, it suffices to focus on the scenario where $x \in \Rint$ and ${F}_{x}\cap\partial R \neq \varnothing$.

    Let $y \in {F}_{x}\cap{\left( \partial R \right)}_{j}$ for $j = 1,\dots,6$. This implies that $y - {e}_{j} = x$. Applying the properties ${e}_{4} = -{e}_{1}$, ${e}_{5} = -{e}_{2}$ and ${e}_{6} = -{e}_{3}$, it follows from the definitions of the discrete Laplace operator and Kronecker delta that:
    \begin{align*}
        \left( {\Delta}_{d} + k^2 \right)\delta_{x,y} &= \sum_{i = 1}^{3}{\left( \delta_{x + {e}_{i},y} + \delta_{x - {e}_{i},y} \right)} - \left( 6 - {k}^{2} \right)\delta_{x,y} \\
        &= \sum_{i = 1}^{3}{\left( \delta_{x + {e}_{i},x + {e}_{j}} + \delta_{x - {e}_{i},x + {e}_{j}} \right)} - \left( 6 - {k}^{2} \right)\delta_{x,x + {e}_{j}} = \delta_{x + {e}_{j},x + {e}_{j}} = 1\,.
    \end{align*}
    By considering the definitions of the difference double-layer potential $W$ and the operator $\mathcal{T}$ given by \eqref{eq:W} and \eqref{eq:T_operator}, respectively, and employing the previous equality, we can establish the following expression for the second equality in \hyperref[lem:VW]{Lemma \ref*{lem:VW}}
    \begin{equation*}
        \begin{aligned}
            \left( {\Delta}_{d} + k^2 \right)\left( \mathcal{T}\cG\left( x - y \right) + \delta_{x,y} \right) &= \left( {\Delta}_{d} + k^2 \right)\left( \cG\left( x;y \right) - \cG\left( x;y - {e}_{j} \right) \right) + \left( {\Delta}_{d} + k^2 \right)\delta_{x,y} \\
            &= \delta_{x,y} - \delta_{x,y - {e}_{j}} + \delta_{y + {e}_{j},y} \\
            &= 0 - 1 + 1 = 0\,.
        \end{aligned}
    \end{equation*}
    Therefore, we can conclude that $\left( {\Delta}_{d} + k^2 \right)W\varphi\left( x \right) = 0$, $x \in \Rint$.
\end{proof}
As a consequence of \hyperref[lem:VW]{Lemma \ref*{lem:VW}}, it can be inferred that
\begin{equation*}
    V\varphi\left( x \right) = \sum_{y \in \partial \Omega}\cG\left( x - y \right)\varphi\left( y \right)\,,\quad \text{for}\ x \in \mathring{\Omega}\,,
\end{equation*}
and
\begin{equation*}
    W\varphi\left( x \right) = \sum_{y \in \partial \Omega}\left( \mathcal{T}\cG\left( x - y \right) + \delta_{x,y} \right)\varphi\left( y \right)\,,\quad \text{for}\ x \in \mathring{\Omega}\,,
\end{equation*}
are radiating solutions to equation \eqref{eq:Helmholtz} for any function $\varphi: \partial \Omega \to \bC$.

Notice that, by our assumption on $\partial \Omega$ we have $\Omega = \mathbb{Z}^2$. Consider a point $y_i$ that intersects several sides of $\partial \Omega$. To reduce the number of numerical computations, we shall choose and fix only one side of the boundary. Let $m$ denote the number of points on $\partial\Omega$. We can represent $\partial \Omega$ as a sequence of points $y_1,y_2,\ldots,y_m$ such that $y_i = y_j$ if and only if $i = j$ for all $1 \leq i,j \leq m$. This allows us to express the difference potential $\widetilde{V}$ in \eqref{eq:V} as a single summand associated with the boundary point $y_i$.

Furthermore, given a function $f$ defined on $\partial\Omega$, we can form a vector $F = {\left( f_1,\dots,f_m \right)}^{\top}$, where $f_i \coloneqq f\left( y_i \right)$ for all $1 \leq i \leq m$. Similarly, for an unknown function $\varphi$ defined on $\partial\Omega$, we can write $\Phi = {\left( \varphi_1,\dots,\varphi_m \right)}^{\top}$, where $\varphi_i \coloneqq \varphi\left( y_i \right)$ for all $1 \leq i \leq m$.

We seek a solution to the {\Prp} in the following form:
\begin{equation}\label{eq:sol1}
    u\left( x \right) = \widetilde{V}\varphi\left( x \right) = \sum_{i = 1}^{m}\cG\left( x - {y}_{i} \right){\varphi}_{i}\,,\quad x \in \mathring{\Omega}\,.
\end{equation}
As demonstrated in the proof of \hyperref[lem:VW]{Lemma \ref*{lem:VW}}, it is possible to establish that $u$ given by \eqref{eq:sol1} is a radiating solution to equation \eqref{eq:Helmholtz}. The solution must also satisfy the boundary conditions specified in \eqref{eq:H3}. Consequently, \eqref{eq:H3} gives rise to the following linear system of boundary equations:
\begin{equation}\label{eq:BS}
    \cH\Phi = F\,,
\end{equation}
here
\begin{equation*}
    \cH =
    \begin{pmatrix}
        \cG\left( {y}_{1} - {y}_{1} \right) & \cG\left( {y}_{1} - {y}_{2} \right) & \cG\left( {y}_{1} - {y}_{3} \right) & \cdots & \cG\left( {y}_{1} - {y}_{m} \right) \\
        \cG\left( {y}_{2} - {y}_{1} \right) & \cG\left( {y}_{2} - {y}_{2} \right) & \cG\left( {y}_{2} - {y}_{3} \right) & \cdots & \cG\left( {y}_{2} - {y}_{m} \right) \\
        \cG\left( {y}_{3} - {y}_{1} \right) & \cG\left( {y}_{3} - {y}_{2} \right) & \cG\left( {y}_{3} - {y}_{3} \right) & \cdots & \cG\left( {y}_{3} - {y}_{m} \right) \\
        \vdots & \vdots & \vdots & \ddots & \vdots \\
        \cG\left( {y}_{m} - {y}_{1} \right) & \cG\left( {y}_{m} - {y}_{2} \right) & \cG\left( {y}_{m} - {y}_{3} \right) & \cdots & \cG\left( {y}_{m} - {y}_{m} \right)
    \end{pmatrix}\,.
\end{equation*}
\begin{lemma}
    The linear system of boundary equations \eqref{eq:BS} admits a unique solution.
\end{lemma}
\begin{proof}
    By virtue of the Rouch\'{e}-Capelli theorem, the unique solvability of the linear system of boundary equations \eqref{eq:BS} is equivalent to the non-existence of non-trivial solutions to the associated homogeneous system, given by
    \begin{equation}\label{eq:hom}
        \cH\Phi = 0\,.
    \end{equation}
    Suppose $\Phi^{*} = {\left( {\varphi}_{1}^{*},\dots,{\varphi}_{m}^{*} \right)}^{\top}$ is a solution to \eqref{eq:hom}. Substituting the solution ${\varphi}^{*}$ into \eqref{eq:sol1}, we obtain
    \begin{equation*}
        u\left( x \right) = \widetilde{V}{\varphi}^{*}\left( x \right) = \sum_{i = 1}^{m}\cG\left( x - {y}_{i} \right){\varphi}_{i}^{*}\,,\quad x \in \mathring{\Omega}\,.
    \end{equation*}
    Since ${\varphi}^{*}$ is a solution to the homogeneous problem \eqref{eq:hom} and satisfies the radiation condition \eqref{eq:radcond}, it follows from \hyperref[th:uni]{Theorem \ref*{th:uni}} that $u \equiv 0$ in $\Omega$. Moreover, at each boundary point ${y}_{i} \in \partial \Omega$, we have
    \begin{equation*}
        0 = {\left( {\Delta}_{d} + {k}^{2} \right)u(x) \Big|}_{x = {y}_{i}} = \left( {\Delta}_{d} + {k}^{2} \right)\widetilde{V}{\varphi}^{*}(x) \Big|_{x = {y}_{i}} = \sum_{i = 1}^{m}\delta_{x,y_i}{\varphi}^{*}_i \Big|_{x = {y}_{i}}  = {\varphi}_{i}^{*}\,,\quad \text{for all}\ 1\leq i\leq m\,.
    \end{equation*}
    Therefore, all solutions to the homogeneous system \eqref{eq:hom} are trivial, and the linear system of boundary equations \eqref{eq:BS} has a unique solution.
\end{proof}
\begin{theorem}
    There exists a unique radiating solution to {\Prp}, which can be expressed as \eqref{eq:sol1}, where $\Phi$ is the unique solution to the system of linear equations \eqref{eq:BS}.
\end{theorem}

\section{Results of numerical computations}\label{sec:numresalts}
\subsection{Description of the procedure for computing the lattice Green's function}\label{subsec:lattice_Green's_func}
The main challenge in numerically evaluating solutions to \eqref{eq:BS} lies in computing the lattice Green's function \eqref{eq:gmn}. To address this, we employ the method devised by Berciu and Cook \cite{Berciu2010}, although various approaches have been developed ({\eg}, refer to \cite{Horiguchi1972,Morita1971}). Applying 8-fold symmetry, it is sufficient to calculate the lattice Green's function $\cG\left( i,j \right)$ for $i \geq j \geq 0$. Following the approach in \cite{Berciu2010}, we introduce the vectors $\cV_{2p} = \tps{\left( \cG\left( 2p,0 \right),\cG\left( 2p - 1,1 \right),\ldots,\cG\left( p,p \right) \right)}$ and $\cV_{2p + 1} = \tps{\left( \cG\left( 2p + 1,0 \right),\cG\left( 2p,1 \right),\ldots,\cG\left( p + 1,p \right) \right)}$. These vectors collect all distinct Green's functions $\cG\left( i,j \right)$ with ``Manhattan distances'' $\left| i \right| + \left| j \right|$ of $2p$ and $2p + 1$, respectively. For any Manhattan distance greater than $1$, the equation
\begin{equation}\label{eq:green_delta_stand_form}
    \left( \Delta_{d} + k^2 \right)\cG\left( x \right) = \delta_{x,0}\,,
\end{equation}
can be expressed in matrix form as
\begin{equation}\label{eq:matrix_reccurence_form}
    {\gamma}_{n}\left( k \right){\cV}_{n} = {\alpha}_{n}\left( k \right){\cV}_{n - 1} + {\beta}_{n}\left( k \right){\cV}_{n + 1}\,,
\end{equation}
where the matrices ${\alpha}_{n}\left( k \right)$, ${\beta}_{n}\left( k \right)$, and ${\gamma}_{n}\left( k \right)$ feature the property of sparsity. Notably, only the dimensions of these matrices depend on $n$ (For details on the dimensions of these matrices, refer to \hyperref[tab:size_and_nnz_sparse_matrices]{Table \ref*{tab:size_and_nnz_sparse_matrices}}, and for the assembly of these sparse matrices, see \hyperref[append:sparse_matrix]{Appendix \ref*{append:sparse_matrix}}). As demonstrated in \cite{Berciu2010}, for any $n \ge 1$, we obtain
\begin{equation}\label{eq:sol._reccur._matrix}
    {\cV}_{n} = {A}_{n}\left( k \right){\cV}_{n - 1}\,,
\end{equation}
where the matrices ${A}_{n}\left( k \right)$ are defined by inserting \eqref{eq:sol._reccur._matrix} into the equation \eqref{eq:matrix_reccurence_form}
\begin{equation*}
    {A}_{n}\left( k \right) = {\left[ {\gamma}_{n}\left( k \right) - {\beta}_{n}\left( k \right){A}_{n+1} \right]}^{-1}{\alpha}_{n}\left( k \right)\,.
\end{equation*}
The matrices can be computed starting from a sufficiently large $N$ with ${A}_{N+1}\left( k \right) = 0$. It is important to note that an exceptional case arises when $k = 2$, leading to a failure of this procedure. In such instances, a more refined ``initial guess'' is required than ${A}_{N+1}\left( k \right) = 0$, as $\mathrm{det}\left( {\gamma}_n\left( k \right) \right) = 0$, and the matrix ${\gamma}_{n}\left( k \right) - {\beta}_{n}\left( k \right){A}_{n + 1}$ is non-invertible. To address this scenario, we set $\tilde{k}^{2} = {k}^{2} + \imath\varepsilon$, where $\varepsilon$ is sufficiently small. Assuming, for significantly large Manhattan distances $N = \left| i \right| + \left| j \right|$ where $N + 1 = 2p$, $p \in \bZ$, there exists an asymptotic dependence: $\cG\left( i + 1,j \right) \approx \cG\left( i,j + 1 \right) = \lambda\left( \tilde{k} \right)\cG\left( i,j \right)$, with $\left| \lambda\left( \tilde{k} \right) \right| < 1$. This implies that all propagators exhibit a uniform decrease as $N$ grows. Upon considering the asymptotic dependence in equation \eqref{eq:green_delta_stand_form}, the ensuing quadratic equation is derived: $2{\lambda}^{2} + \left( \tilde{k}^{2} - 4 \right)\lambda + 2 = 0$, yielding the immediate conclusion of the existence of a physical solution $\left| \lambda\left( \tilde{k} \right) \right| < 1$ ({\cf} \cite{Berciu2010}). The resulting matrix ${A}_{N+1}\left( k \right)$ has a dimension $\left( p + 1 \right) \times p$, and its entries are defined as follows: ${A}_{N+1}\left( \tilde{k} \right)\mid_{i,i} = {A}_{N+1}\left( \tilde{k} \right)\mid_{i,i - 1} = \lambda / 2$, $i = \overline{2,p}$. Additionally, ${A}_{N+1}\left( \tilde{k} \right)\mid_{1,1} = {A}_{N+1}\left( \tilde{k} \right)\mid_{p + 1,p} = {\lambda}$, and all other matrix elements remain zero.

Once ${A}_{n}\left( k \right)$ are known, we can express ${\cV}_{n} = {A}_{n}\left( k \right) \cdots {A}_{1}\left( k \right){\cV}_{0}$, where ${\cV}_{0} = \cG\left( 0,0 \right)$. Specifically, ${\cV}_{1} = \cG\left( 1,0 \right) = {A}_{1}\left( k \right)\cG\left( 0,0 \right)$. By inserting this into equation \eqref{eq:green_delta_stand_form} and considering property \eqref{eq:property_Green's_funct.}, we obtain the expression $6\cG\left( 1,0 \right) - \left( 6 - {k}^{2} \right)\cG\left( 0,0 \right) = 1$. From this, we determine that $\cG\left( 0,0 \right) = 1 / \left[ 6 {A}_{1}\left( k \right) - 6 + {k}^{2} \right]$. This concludes the computation employing only elementary operations and devoid of integrals. Additionally, it is important to note another meaningful advantage of this method. The matrices ${A}_{n}\left( k \right)$ are computed by descending from asymptotically large Manhattan distances. As they propagate towards diminished Manhattan distances, it provides us with the physical solution.

Let us proceed with developing another approach to determine the ``initial guess'' for computing the lattice Green's function using the described procedure, particularly when $k$ pertains to the set of real numbers, specifically $k \in \left( 0,2\sqrt{2} \right)$. Assuming asymptotically large Manhattan distances $\left| i \right| + \left| j \right| \gg 1$, the following heuristic relation holds:
\begin{equation*}
    \cG\left( r,s \right) = {\lambda}^{\left( r + s \right) - \left( i + j \right)} q\left( {r + s},{i + j} \right) {h} \cG\left( i,j \right)\,,\quad q\left( {r + s},{i + j} \right) = \frac{i + j}{r + s}\,.
\end{equation*}
The role of the parameters $\lambda$ and $h \neq 0$ are defined in the subsequent discussion. According to \eqref{eq:sol._reccur._matrix}, we express ${\cV}_{n + 1} = {A}_{n + 1}\left( k \right){\cV}_{n}$, where $n + 1 = 2p$. Consequently, we obtain the following representations for the vectors ${\cV}_{n + 1}$ and ${\cV}_{n}$, {\viz}, ${\cV}_{n + 1} = \tps{\left( \cG\left( 2p,0 \right),\cG\left( 2p - 1,1 \right),\ldots,\cG\left( p,p \right) \right)}$ and $\cV_{n} = \tps{\left( \cG\left( 2p - 1,0 \right),\cG\left( 2p - 2,1 \right),\ldots,\cG\left( p,p - 1 \right) \right)}$. Let us now express the relation between ${\cV}_{n + 1}$ and ${\cV}_{n}$ in matrix-vector form
\begin{equation*}
    {
    \begin{pmatrix}
        \cG\left( 2p,0 \right) \\
        \cG\left( 2p - 1,1 \right) \\
        \cG\left( 2p - 2,2 \right) \\
        \vdots \\
        \cG\left( p,p \right)
    \end{pmatrix}
    }_{\left( p + 1 \right) \times 1}
    =
    {\left[ {A}_{n + 1}\left( k \right) \right]}_{\left( p + 1 \right) \times p}
    {
    \begin{pmatrix}
        \cG\left( 2p - 1,0 \right) \\
        \cG\left( 2p - 2,1 \right) \\
        \vdots \\
        \cG\left( p,p - 1 \right)
    \end{pmatrix}
    }_{p \times 1}\,.
\end{equation*}
In our cases, we have $i = 2p - \ell$ and $j = \ell$, where $\ell$ ranges from $0$ to $p$. For simplicity, we reintroduce the notation $q\left( {2p + r + s},{2p} \right)$ as $q\left( p,{r + s} \right)$, meaning $q\left( {2p + r + s},{2p} \right) = q\left( p,{r + s} \right)$. As a result, we can easily derive
\begin{equation}\label{eq:assumption_G}
    \cG\left( 2p - \ell + r,\ell + s \right) = {\lambda}^{r + s} q\left( p,{r + s} \right) {h} \cG\left( 2p - \ell,\ell \right)\,,\quad q\left( p,{r + s} \right) = \frac{2p}{2p + r + s}\,.
\end{equation}
Due to the equation \eqref{eq:green_delta_stand_form}, we obtain
\begin{align*}
    \left( 6 - {k}^{2} \right)\cG\left( 2p - \ell,\ell \right) &= \cG\left( 2p - \ell + 1,\ell \right) + \cG\left( 2p - \ell - 1,\ell \right) \\
    &+ \cG\left( 2p - \ell,\ell + 1 \right) + \cG\left( 2p - \ell,\ell - 1 \right) \\
    &+ \cG\left( 2p - \ell + 1,\ell - 1 \right) + \cG\left( 2p - \ell - 1,\ell + 1 \right)\,,
\end{align*}
and, considering the assumption for $\cG\left( 2p - \ell + r,\ell + s \right)$, we arrive at the following equation:
\begin{equation}\label{eq:quad_poly_lambda}
    a\left( p,h \right){\lambda}^{2} + b\left( k,h \right)\lambda + c\left( p,h \right) = 0\,,
\end{equation}
where
\begin{align*}
    a\left( p,h \right) &= 2 q\left( p,{1} \right) {h} = \frac{4p}{2p + 1} {h}\,, \\
    b\left( k,h \right) &= {k}^{2} - 6 + 2 q\left( p,{0} \right) {h} = {k}^{2} - 6 + 2{h}\,, \\
    c\left( p,h \right) &= 2 q\left( p,{-1} \right) {h} = \frac{4p}{2p - 1} {h}\,.
\end{align*}
The discriminant of the quadratic polynomial \eqref{eq:quad_poly_lambda} is given by
\begin{equation}\label{eq:discr_lambd_poly}
    \mathcal{D}_{1} = {\left( {k}^{2} - 6 + 2{h} \right)}^{2} - \frac{64{p}^{2}}{4{p}^{2} - 1}{h}^{2} = -4\frac{12{p}^{2} + 1}{4{p}^{2} - 1}{h}^{2} + 4\left( {k}^{2} - 6 \right){h} + {\left( {k}^{2} - 6 \right)}^{2}\,.
\end{equation}
Let us now calculate the discriminant of the quadratic polynomial \eqref{eq:discr_lambd_poly}, that is
\begin{equation}\label{eq:discr_h_poly}
    \frac{\mathcal{D}_{2}}{4} = \frac{64{p}^{2}}{4{p}^{2} - 1}{\left( {k}^{2} - 6 \right)}^{2} \geq 0\,,\quad p \gg 1\,.
\end{equation}
The solutions to the quadratic polynomial \eqref{eq:discr_lambd_poly} are expressed in the following manner:
\begin{equation}\label{eq:solution_quad_eqt_h}
    \begin{aligned}
        h =
        \left\{
        \begin{array}{ll}
           \left( 6 - {k}^{2} \right){\left( \pm \frac{8p}{\sqrt{4{p}^{2} - 1}} + 2 \right)}^{-1},  & \hbox{\text{for} $0 < {k}^{2} < 6$,} \\
           \left( {k}^{2} - 6 \right){\left( \pm \frac{8p}{\sqrt{4{p}^{2} - 1}} - 2 \right)}^{-1},  & \hbox{\text{for} $6 < {k}^{2} < 8$.}
        \end{array}
        \right.
    \end{aligned}
\end{equation}

To establish a relationship between $\lambda$ and $h$, it is essential to align with the behaviour of the solution to the given problem. Specifically, wave propagation diminishes over sufficiently large Manhattan distances. In simpler terms, all propagators decrease at a certain rate as $n$ increases. We express this fact in the following way, observing that due to \eqref{eq:assumption_G}, we have the asymptotic dependence
\begin{equation*}
    \cG\left( 2p - \ell + 1,\ell \right) = \cG\left( 2p - \ell,\ell + 1 \right) = {\lambda} \frac{2p}{2p + 1} {h} \cG\left( 2p - \ell,\ell \right)\,,\quad h \neq 0\,.
\end{equation*}
Following this point, the subsequent condition must be fulfilled
\begin{equation}\label{eq:relat_lambda_and_h}
    \left| \lambda \right| \frac{2p}{2p + 1} \left| {h} \right| < 1\,,
\end{equation}
this implies that
\begin{equation}\label{eq:lambda_h_inequality}
    \left| h \right| < \frac{1}{\left| {\lambda} \right|}\frac{2p + 1}{2p} = \frac{1}{\left| {\lambda} \right|}\left( 1 + \frac{1}{2p} \right)\,.
\end{equation}

In the specific case where ${k}^{2} = 6$, as assumed, and with the additional condition that $h \neq 0$, it follows from \eqref{eq:discr_lambd_poly} that $\mathcal{D}_{1} < 0$. Consequently, the quadratic equation \eqref{eq:quad_poly_lambda} possesses two distinct (non-real) complex roots, which are complex conjugates of each other. In this case, we derive
\begin{equation}\label{eq:specific_case_k}
    \left| \lambda \right| = \sqrt{\frac{2p + 1}{2p - 1}}\,.
\end{equation}
Considering the upper bound \eqref{eq:lambda_h_inequality} for $\left| \lambda \right|$, along with the condition $h \neq 0$ and the quality \eqref{eq:specific_case_k}, we easily obtain the result
\begin{equation*}
    0 < \left| {h} \right| < \sqrt{1 - \frac{1}{4{p}^{2}}}\,,\quad p \gg 1\,.
\end{equation*}
Let us now move forward to estimate $h$ in the case where ${k}^{2} \in \left( 0,8 \right) \setminus \left\{ 6 \right\}$. In this context, consider the scenario where the discriminant \eqref{eq:discr_lambd_poly} of the quadratic polynomial \eqref{eq:quad_poly_lambda} is negative, meaning that $h$ lies outside the roots of the polynomial \eqref{eq:discr_lambd_poly}. Consequently, for the polynomial \eqref{eq:quad_poly_lambda}, we have two distinct complex roots ${\lambda}_{1}$ and ${\lambda}_{2}$. It is easy to show, for \eqref{eq:solution_quad_eqt_h}, that the following estimate is valid
\begin{equation}\label{eq:abs_h_lower_bound}
    \left| h \right| \geq \frac{\left| {k}^{2} - 6 \right|}{6}\,.
\end{equation}
The same equality is derived for $\left| \lambda \right|$ as \eqref{eq:specific_case_k}. Considering this result together with \eqref{eq:lambda_h_inequality} and \eqref{eq:abs_h_lower_bound}, we obtain
\begin{equation*}
    \frac{\left| {k}^{2} - 6 \right|}{6} \leq \left| h \right| < \sqrt{1 - \frac{1}{4{p}^{2}}}\,,
\end{equation*}
providing that
\begin{equation*}
    p > \frac{3}{k\sqrt{12 - {k}^{2}}}\,.
\end{equation*}

Before presenting the entries of the matrix ${A}_{n + 1}\left( k \right)$, we first provide a remark regarding $\lambda$, which is the root of the quadratic polynomial \eqref{eq:quad_poly_lambda} and depends on the values of $k$, $p$, and $h$. The entries of the matrix ${A}_{n + 1}\left( k \right)$ are denoted using the following notation: ${a}_{i,j}^{n + 1}\left( k \right)$, where $i = 1,2,\ldots,p+1$ and $j = 1,2,\ldots,p$. From the matrix-vector form, we easily identify that non-zero elements of the matrix ${A}_{n + 1}\left( k \right)$ are defined as follows:
\begin{equation*}
    {a}_{1,1}^{n + 1}\left( k \right) = {a}_{p + 1,p}^{n + 1}\left( k \right) = \frac{2p - 1}{2p} {\rho}\,, \quad {a}_{i,i - 1}^{n + 1}\left( k \right) = {a}_{i,i}^{n + 1}\left( k \right) = \frac{2p - 1}{4p} {\rho}\,,\quad {\rho} = \frac{\lambda}{h}\,,\quad i = 2,3,\ldots,p\,.
\end{equation*}

\subsection{Numerical experiments on computing lattice Green's functions}\label{subsec:comp_Green's_func}
Prior to raising the numerical results of the given problem \eqref{eq:Helmholtz}-\eqref{eq:H3}, it is advisable to provide remarks regarding the numerical computations of the lattice Green's function \eqref{eq:gmn}. The accuracy of the calculation of the exact solution representation \eqref{eq:sol1} mainly depends on the numerical approximation of the lattice Green's function. To demonstrate the experimental convergence of the procedure outlined in \hyperref[subsec:lattice_Green's_func]{Subsec. \ref*{subsec:lattice_Green's_func}} for the numerical computation of the lattice Green's function, consider ${p}_{m} = {2}^{m} {p}_{0}$ with ${p}_{0} = 71$, and ${N}_{m} = {2}^{m + 1} {p}_{0} - 1$, where $m = 1,2,3,4$. For each step $m$, we define the following vectors: ${\cV}_{2i}^{\left( m \right)} = \tps{\left( {\cG}^{\left( m \right)}\left( 2i,0 \right),{\cG}^{\left( m \right)}\left( 2i - 1,1 \right),\ldots,{\cG}^{\left( m \right)}\left( i,i \right) \right)}$ and ${\cV}_{2i + 1}^{\left( m \right)} = \tps{\left( {\cG}^{\left( m \right)}\left( 2i + 1,0 \right),{\cG}^{\left( m \right)}\left( 2i,1 \right),\ldots,{\cG}^{\left( m \right)}\left( i + 1,i \right) \right)}$, where $i = 0,1,\ldots,{p}_{m} - 1$. It is important to note that, for every step $m$ starting from $0$, there are a total of ${N}_{m} + 1$ such vectors, each with the dimension of $\left( i + 1 \right) \times 1$. As $m$ increases, the first ${N}_{0} + 1 = 2 {p}_{0}$ vectors are common. To collect the elements of these common vectors, it is useful to form a matrix array by completing them. Consider the matrix $\boldsymbol{\cV}_{{N}_{0}}^{\left( m \right)}$ of dimension $\left( {N}_{0} + 1 \right) \times {p}_{0}$, with its entries denoted by ${v}_{i,j}^{\left( m \right)}$. The non-zero elements of this matrix are precisely defined as ${v}_{i,j}^{\left( m \right)} = {\cG}^{\left( m \right)}\left( i - 1,j - 1 \right)$, subject to the constraints $i - j \geq 0$ and $i + j \leq {N}_{0} + 2$, which correspond to ${p}_{0}\left( {p}_{0} + 1 \right) = 5112$ entries. All other entries of $\boldsymbol{\cV}_{{N}_{0}}^{\left( m \right)}$ are zero. Let us consider the element-wise norm of the matrix difference, restricted by $j \leq i \leq -j + {N}_{0} + 2$ and $j = 1,2,\ldots,{p}_{0}$, defined as:
\begin{equation*}
    {\left\| \boldsymbol{\cV}_{{N}_{0}}^{\left( m \right)} - \boldsymbol{\cV}_{{N}_{0}}^{\left( m - 1 \right)} \right\|}_{\max} = \max\limits_{i,j} \left| {v}_{i,j}^{\left( m \right)} - {v}_{i,j}^{\left( m - 1 \right)} \right|\,,\quad m = 1,2,3,4\,.
\end{equation*}

By employing analogous reasoning, we introduce the matrix $\boldsymbol{\cV}_{{N}_{0},{h}}^{\left( m \right)}$, wherein the non-zero entries correspond to the lattice Green's function computed using another choice of the ``initial guess'' that depends on the parameter $h$. In our computations, we explored various values of $\left| h \right|$, resulting in modifications to the coefficients associated with the convergence order. Nevertheless, the convergence order itself remains consistent with that observed in all preceding cases. In the considered scenario, we select $\left| h \right| = 1 / \left| \lambda \right| = \sqrt{(2{p}_{m} - 1) / (2{p}_{m} + 1)}$. The results derived from the described computations for the wavenumber $k = 2$ are summarized in \hyperref[tab:Experim_Converg]{Table \ref*{tab:Experim_Converg}}. Additionally, \hyperref[fig:real_G_funcs]{Figure \ref*{fig:real_G_funcs}} provides graphical representations of the real part of the lattice Green's functions for the triangular lattice in both transformed and original coordinates, corresponding to ${N}_{4} = 2271$. For illustrative purposes, the transformed coordinates of the triangular lattice are plotted within the range $\left[ -40,40 \right] \times \left[ -40,40 \right] \subset {\bZ}^{2}$. Meanwhile, the original coordinates are confined to a subset of ${\bR}^{2}$ through a suitable inverse transformation from ${\bZ}^{2}$ coordinates.

It is important to highlight that all calculations were performed using the scientific programming language GNU Octave $8.4.0$ on a standard laptop computer equipped with an AMD Ryzen $5$ $5600$H processor with Radeon Graphics ($12$ CPUs), running at approximately $3.3$GHz, and having $8192$MB RAM. We encountered limitations preventing further exploration ({\ie}, for $m \geq 5$). Specifically, for $m = 5$, the computation involves evaluating ${p}_{5}\left( {p}_{5} + 1 \right) = \pgfmathprintnumber[std,std=-1:0,fixed zerofill,sci zerofill,sci e,precision=6,1000 sep={}]{5.164256000000000e+06}$ lattice Green's function values, exceeding the available memory capacity of this machine.
\begin{table}[H]
    \centering


\begin{tabular}{ccccc}
    \toprule
    $m$ & ${p}_{m}$ & ${N}_{m}$ & ${\left\| \boldsymbol{\cV}_{{N}_{0}}^{\left( m \right)} - \boldsymbol{\cV}_{{N}_{0}}^{\left( m - 1 \right)} \right\|}_{\max}$ & ${\left\| \boldsymbol{\cV}_{{N}_{0},{h}}^{\left( m \right)} - \boldsymbol{\cV}_{{N}_{0},{h}}^{\left( m - 1 \right)} \right\|}_{\max}$ \\
    \midrule
    $0$ & $71$ & $141$ & $-$ & $-$ \\
    $1$ & $142$ & $283$ & $\pgfmathprintnumber[std,std=-1:0,fixed zerofill,sci zerofill,sci e,precision=4,1000 sep={}]{9.154631543183475e-02}$ & $\pgfmathprintnumber[std,std=-1:0,fixed zerofill,sci zerofill,sci e,precision=4,1000 sep={}]{9.082118795036294e-02}$ \\
    $2$ & $284$ & $567$ & $\pgfmathprintnumber[std,std=-1:0,fixed zerofill,sci zerofill,sci e,precision=4,1000 sep={}]{8.760788431425402e-04}$ & $\pgfmathprintnumber[std,std=-1:0,fixed zerofill,sci zerofill,sci e,precision=4,1000 sep={}]{8.263699031209913e-04}$ \\
    $3$ & $568$ & $1135$ & $\pgfmathprintnumber[std,std=-1:0,fixed zerofill,sci zerofill,sci e,precision=4,1000 sep={}]{4.119278606499520e-04}$ & $\pgfmathprintnumber[std,std=-1:0,fixed zerofill,sci zerofill,sci e,precision=4,1000 sep={}]{4.252963846704570e-04}$ \\
    $4$ & $1136$ & $2271$ & $\pgfmathprintnumber[std,std=-1:0,fixed zerofill,sci zerofill,sci e,precision=4,1000 sep={}]{1.807911326154362e-04}$ & $\pgfmathprintnumber[std,std=-1:0,fixed zerofill,sci zerofill,sci e,precision=4,1000 sep={}]{1.772844979442056e-04}$ \\
    \bottomrule
\end{tabular}
    \caption{The absolute error associated with the difference between the values of every two successive lattice Green's functions for the case of $k = 2$.}
    \label{tab:Experim_Converg}
\end{table}
We conducted numerical experiments, presented in \hyperref[tab:Experim_Converg]{Table \ref*{tab:Experim_Converg}}, to investigate the behaviour of computed lattice Green's functions when initiating computations with various fixed Manhattan distances ${N}_{m}$. Our analysis focused on observing the maximum absolute differences between lattice Green’s functions that are common in all considered instances. The numerical results indicate a stability property when computations start from a considerable distance. As referred in \cite{Berciu2010}, initiating computations for asymptotically large Manhattan distances with a refined ``initial guess'' yields an immediate physical solution. However, advancing further in the context of the subsequent increase in Manhattan distance established impossible for us, given the limitations of the machine resources.

In addition to the prescribed approach for computing lattice Green's functions, we have explored various numerical integration techniques, including the composite trapezoidal rule, composite Simpson's $1/3$ rule, and Gaussian quadrature. These methods are employed for the numerical approximation of the lattice Green's functions \eqref{eq:gmn} with ${k}^{2} + \imath\varepsilon$, which involve double integrals. It is observed that the integrand in \eqref{eq:gmn} exhibits fast oscillations as the magnitudes of ${x}_{1}$ and ${x}_{2}$ increase from the origin. To address this behaviour, it becomes necessary to reduce the mesh length for both integration variables significantly. However, this adjustment alone is insufficient to achieve the desired precision, as the fast oscillations function is approximated. Additionally, it becomes computationally expensive.

\begin{figure}[H]
    \centering
    \begin{subfigure}{.49\textwidth}
        \centering
        \includegraphics[width=\textwidth,height=\textheight,keepaspectratio]{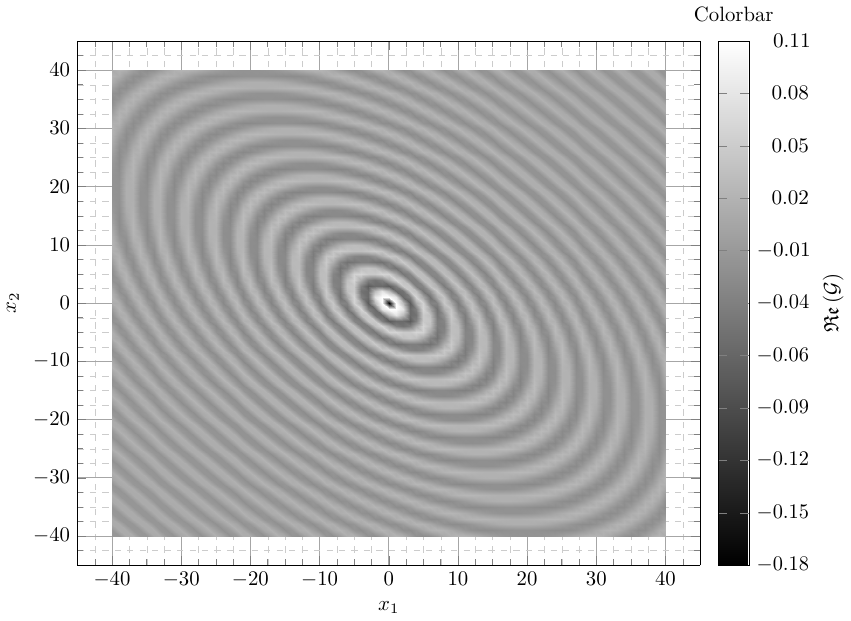}
        \caption{The density plot of $\Real{\cG}$; $\left( {x}_{1},{x}_{2} \right) \in {\bZ}^2$.}
        \label{fig:real_lattice_G_funcs_cartesian}
    \end{subfigure}
    \hfill
    \begin{subfigure}{.49\textwidth}
        \centering
        \includegraphics[width=\textwidth,height=\textheight,keepaspectratio]{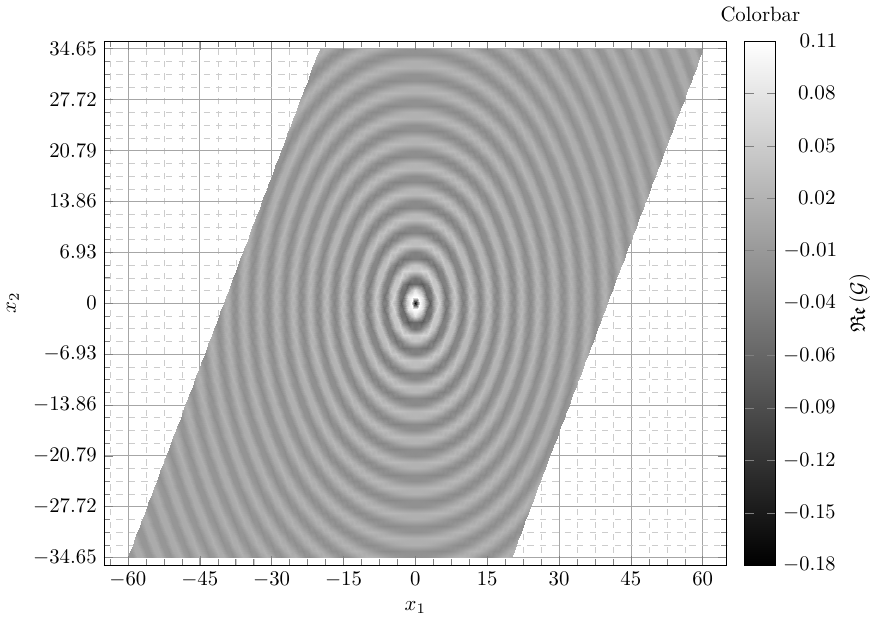}
        \caption{The density plot of $\Real{\cG}$; $\left( {x}_{1},{x}_{2} \right) \in {\bR}^2$.}
        \label{fig:real_lattice_G_funcs_triangular}
    \end{subfigure}
    \caption{The plots of the real part of the lattice Green's functions in sub-figures (a) and (b) are presented in both the transformed and original coordinates of the triangular lattice, respectively, wherein $k = 2$.}
    \label{fig:real_G_funcs}
\end{figure}

\subsection{Numerical experiments for various instances}\label{subsec:num_exper}
Consider {\Prp} with $\partial\Omega = {\Gamma}_{1} \cup {\Gamma}_{2}$. Specifically, we have ${\Gamma}_{1} = \left\{ \left( -5,0 \right), \left( -4,0 \right) \right\}$ and ${\Gamma}_{2} = \left\{ \left( 4,0 \right), \left( 5,0 \right) \right\}$. In this problem, two distinct cases are considered:
\begin{itemize}
    \item In the first case, referred to as the symmetric mode, we assume that the function $f\left( y \right) \equiv 1$ is constant on both line segments. This assumption is made for simplicity.
    \item In the second case, referred to as the skew-symmetric mode, we assume that $f\left( y \right) \equiv -1$ on ${\Gamma}_{1}$, while $f\left( y \right) \equiv 1$ on ${\Gamma}_{2}$.
\end{itemize}
The boundary $\partial \Omega$ consists of four points: ${y}_{1} = \left( -5,0 \right)$, ${y}_{2} = \left( -4,0 \right)$, ${y}_{3} = -{y}_{2} = \left( 4,0 \right)$ and ${y}_{4} = -{y}_{1} = \left( 5,0 \right)$.

The vector $\Phi = \tps{\left( \varphi_1,\dots,\varphi_4 \right)}$ is the unique solution to equation \eqref{eq:BS}, where the coefficient matrix $\cH$ of the system \eqref{eq:BS} is symmetric. Specifically, in this context, $\cH$ is given by:
\begin{equation*}
    \cH =
    \begin{pmatrix}
        \cG\left( 0,0 \right) & \cG\left( 1,0 \right) & \cG\left( 9,0 \right) & \cG\left( 10,0 \right) \\
        \cG\left( 1,0 \right) & \cG\left( 0,0 \right) & \cG\left( 8,0 \right) & \cG\left( 9,0 \right) \\
        \cG\left( 9,0 \right) & \cG\left( 8,0 \right) & \cG\left( 0,0 \right) & \cG\left( 1,0 \right) \\
        \cG\left( 10,0 \right) & \cG\left( 9,0 \right) & \cG\left( 1,0 \right) & \cG\left( 0,0 \right)
    \end{pmatrix}\,.
\end{equation*}
In order to solve the obtained system of linear equations \eqref{eq:BS} and find the solution $u$, we have developed GNU Octave code implementing an efficient method outlined previously for computing lattice Green’s functions. Notably, these computations were completed within several minutes on a standard personal computer. For our analyses, we fix the wavenumber at $k = 2$, and truncate the computation of lattice Green's functions for $N = 2271$.

The results of numerical evaluations are presented in \hyperref[fig:real_solution]{Figure \ref*{fig:real_solution}} and \hyperref[fig:abs_solution]{Figure \ref*{fig:abs_solution}}. In \hyperref[fig:real_solution]{Figure \ref*{fig:real_solution}}, each sub-figure depicts a density plot of $\Real{u}$. Sub-figures (a) and (b) correspond to the symmetric mode, while sub-figures (c) and (d) represent the skew-symmetric mode in the coordinates of ${\bZ}^2$ and ${\bR}^2$, respectively. \hyperref[fig:abs_solution]{Figure \ref*{fig:abs_solution}} displays the density plot of $\left| u \right|$, with sub-figures (a) and (b) illustrating the symmetric mode, and sub-figures (c) and (d) depicting the skew-symmetric mode in the respective coordinate systems.

Several key features of the numerical solutions are apparent. As expected, both $\Real{u}$ and $\left| u \right|$ exhibit symmetry, as evidenced by these figures. Moreover, we observe the interference phenomenon of waves.

\begin{figure}[H]
    \centering
    \begin{subfigure}{.49\textwidth}
        \centering
        \includegraphics[width=\textwidth,height=\textheight,keepaspectratio]{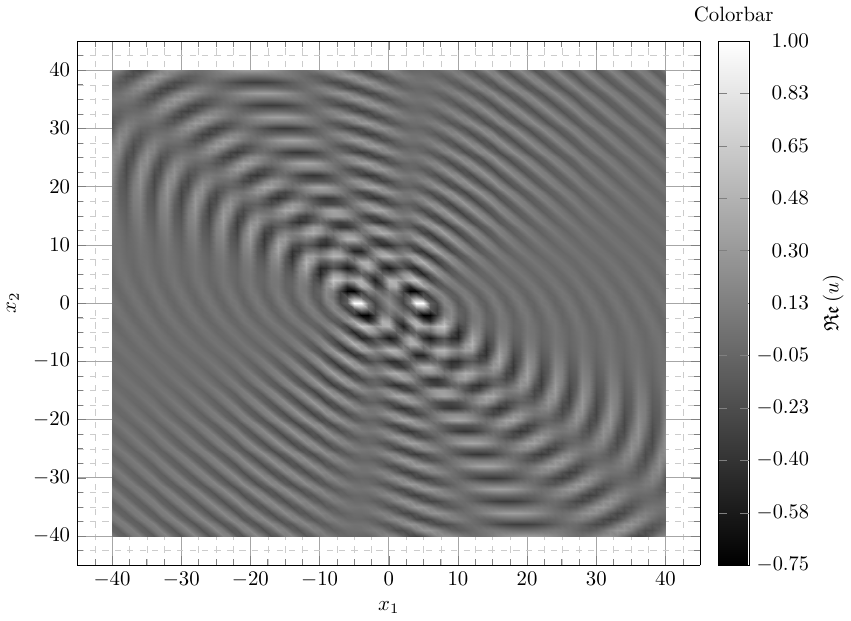}
        \caption{Symmetric mode: the density plot of $\Real{u}$.}
        \label{fig:sym_mode_Z_sq_real}
    \end{subfigure}
    \hfill
    \begin{subfigure}{.49\textwidth}
        \centering
        \includegraphics[width=\textwidth,height=\textheight,keepaspectratio]{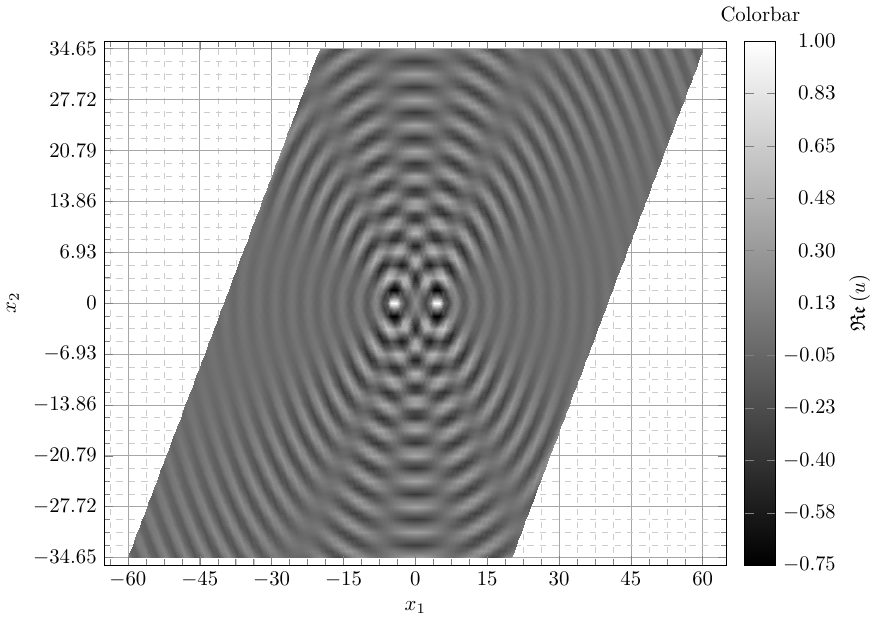}
        \caption{Symmetric mode: the density plot of $\Real{u}$.}
        \label{fig:sym_mode_R_sq_real}
    \end{subfigure}
    \hfill
    \begin{subfigure}{.49\textwidth}
        \centering
        \includegraphics[width=\textwidth,height=\textheight,keepaspectratio]{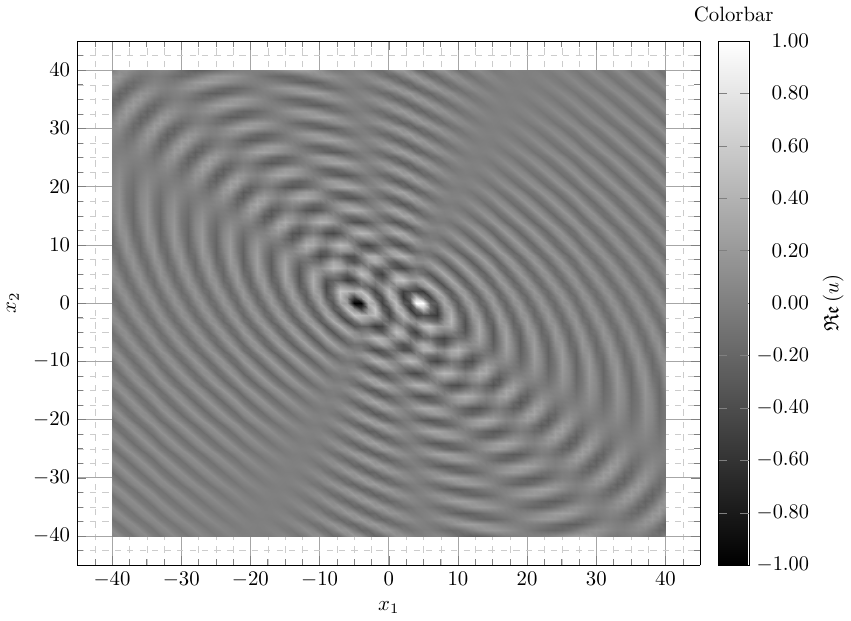}
        \caption{Skew-symmetric mode: the density plot of $\Real{u}$.}
        \label{fig:skew-sym_mode_Z_sq_real}
    \end{subfigure}
    \hfill
    \begin{subfigure}{.49\textwidth}
        \centering
        \includegraphics[width=\textwidth,height=\textheight,keepaspectratio]{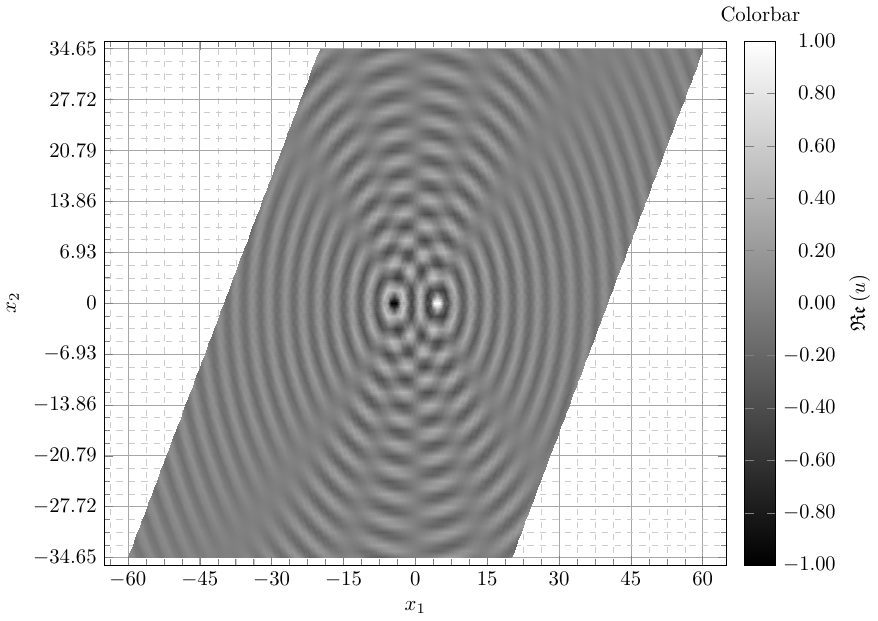}
        \caption{Skew-symmetric mode: the density plot of $\Real{u}$.}
        \label{fig:skew-sym_mode_R_sq_real}
    \end{subfigure}
    \caption{The plots shown in (a) and (c) are represented in ${\bZ}^2$, whereas plots (b) and (d) are graphed on the original coordinates of the triangular lattice, where $k = 2$.}
    \label{fig:real_solution}
\end{figure}

\begin{figure}[H]
    \centering
    \begin{subfigure}{.49\textwidth}
        \centering
        \includegraphics[width=\textwidth,height=\textheight,keepaspectratio]{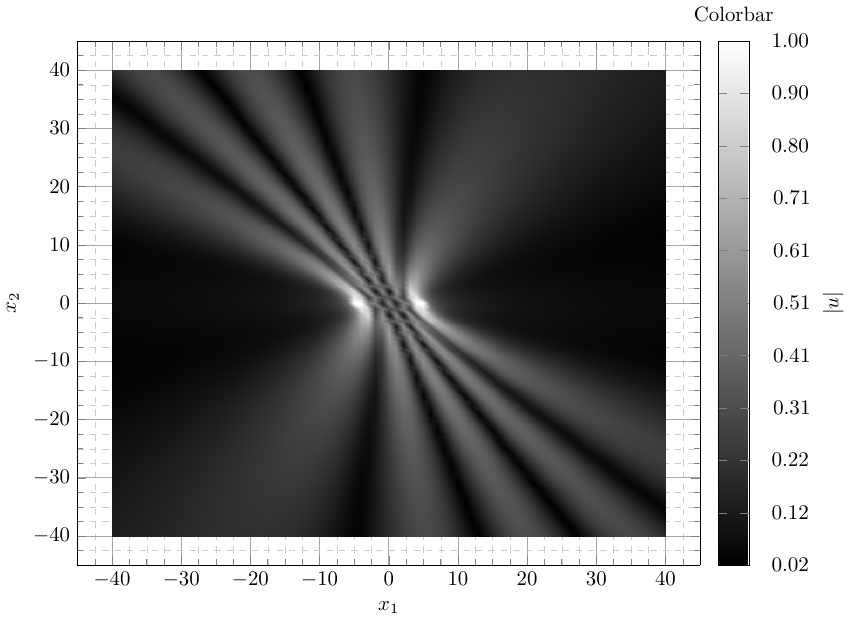}
        \caption{Symmetric mode: the density plot of $\left| u \right|$.}
        \label{fig:fig:sym_mode_Z_sq_abs}
    \end{subfigure}
    \hfill
    \begin{subfigure}{.49\textwidth}
        \centering
        \includegraphics[width=\textwidth,height=\textheight,keepaspectratio]{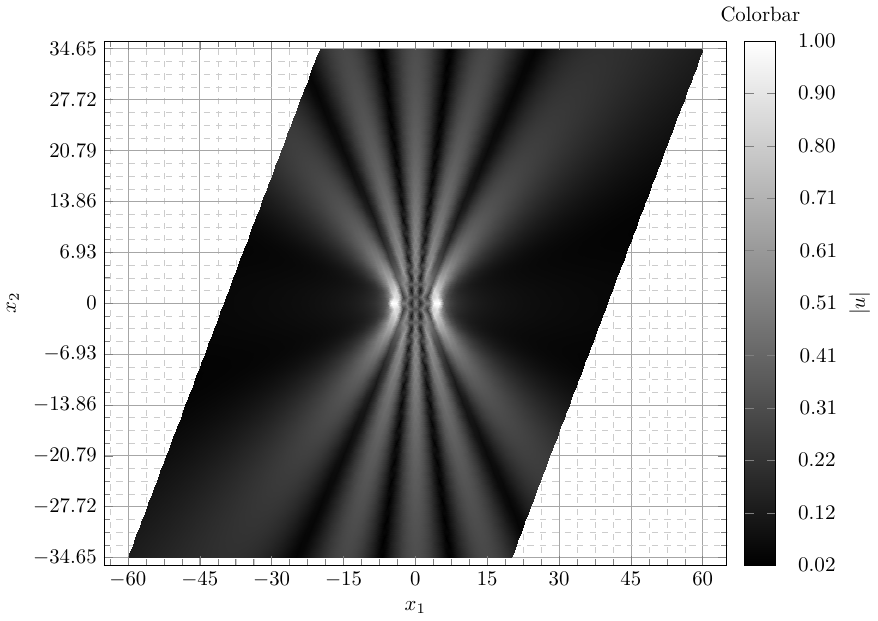}
        \caption{Symmetric mode: the density plot of $\left| u \right|$.}
        \label{fig:fig:sym_mode_R_sq_abs}
    \end{subfigure}
    \hfill
    \begin{subfigure}{.49\textwidth}
        \centering
        \includegraphics[width=\textwidth,height=\textheight,keepaspectratio]{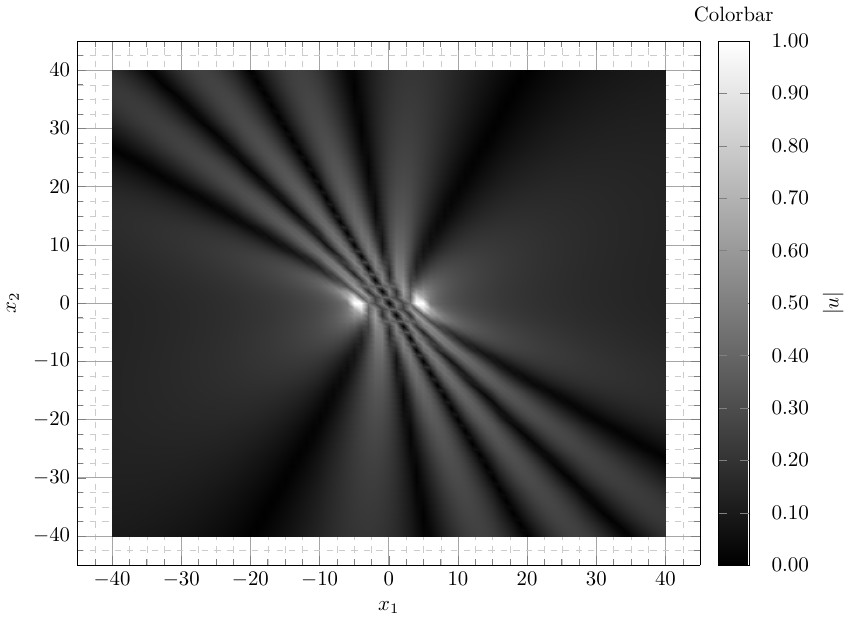}
        \caption{Skew-symmetric mode: the density plot of $\left| u \right|$.}
        \label{fig:skew-sym_mode_Z_sq_abs}
    \end{subfigure}
    \hfill
    \begin{subfigure}{.49\textwidth}
        \centering
        \includegraphics[width=\textwidth,height=\textheight,keepaspectratio]{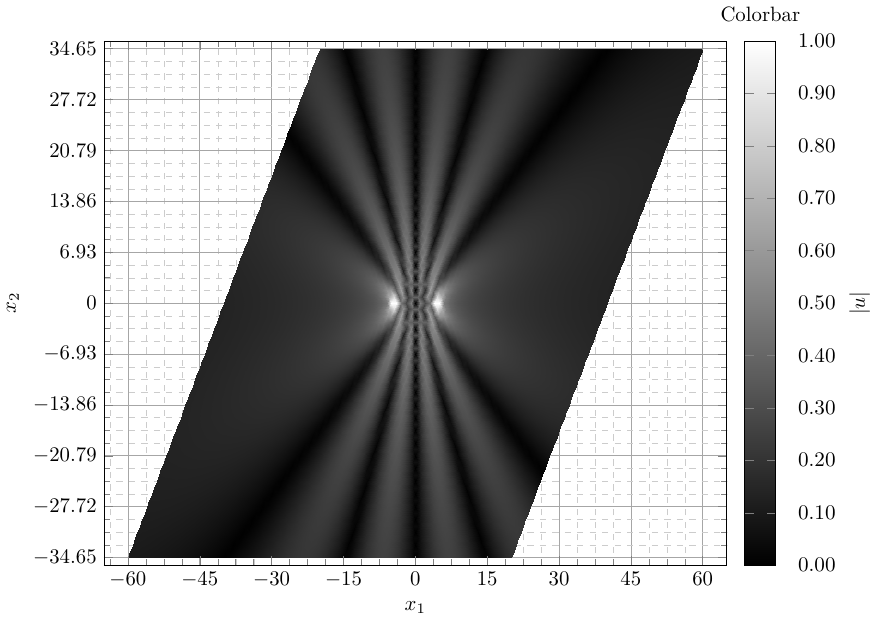}
        \caption{Skew-symmetric mode: the density plot of $\left| u \right|$.}
        \label{fig:skew-sym_mode_R_sq_abs}
    \end{subfigure}
    \caption{The plots displayed in (a) and (c) are presented in the coordinate system of ${\bZ}^2$, whereas plots (b) and (d) are plotted on the initial coordinates of the triangular lattice, where the value of $k$ is equal to $2$.}
    \label{fig:abs_solution}
\end{figure}

Finally, we would like to emphasize that we rely on the cone condition requirement to prove the uniqueness result when applying the unique continuation property. However, this requirement is not always necessary, and we can effectively use the unique continuation property for certain configurations. An example of such a configuration is presented in the following problem:

Consider the problem denoted {\Prp}, where the boundary $\partial\Omega$ is defined as the union of two disjoint sets: ${\Gamma}_{1}$ and ${\Gamma}_{2}$. Specifically, we fix ${\Gamma}_{1} = \left\{ \left( -3,1 \right), \left( -2,1 \right), \left( -1,1 \right), \left( 0,1 \right), \left( 1,1 \right) \right\}$ and ${\Gamma}_{2} = \left\{ \left( -2,-1 \right), \left( -1,-1 \right), \left( 0,-1 \right), \left( 1,-1 \right), \left( 2,-1 \right) \right\}$. The boundary $\partial \Omega$ comprises ten discrete points, herein represented by ${y}_{i}$ ($i = 1,2,\ldots,10$), depicted as red rhombi in \hyperref[fig:examp2_ten_points]{Figure \ref*{fig:examp2_ten_points}}. It is assumed that the function $f\left( {y}_{i} \right) \equiv 1$ remains constant along both parallel line segments. In this case, the coefficient matrix $\cH$ of a linear system of boundary equations \eqref{eq:BS} has dimensions $10 \times 10$. Following the approach used in the previous example, we solve this system using a standard GNU Octave routine. We set the wavenumber to $k = 2$, and truncate the computation of lattice Green's functions for $N = 2271$.
\begin{figure}[H]
    \centering
    \includegraphics[width=0.40\textwidth]{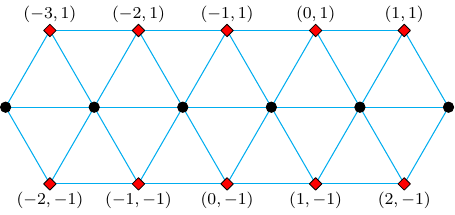}
    \caption{Arrangement of ten discrete boundary points on parallel line segments within a triangular lattice.}
    \label{fig:examp2_ten_points}
\end{figure}

In this case, zeros ({\cf}, end of the proof of \hyperref[th:uni]{Theorem \ref*{th:uni}}) propagate freely between parallel boundaries, yet the uniqueness result persists. The results of numerical evaluations are presented in \hyperref[fig:examp02_real_u]{Figure \ref*{fig:examp02_real_u}}.

\begin{figure}[H]
    \centering
    \begin{subfigure}{.49\textwidth}
        \centering
        \includegraphics[width=\textwidth,height=\textheight,keepaspectratio]{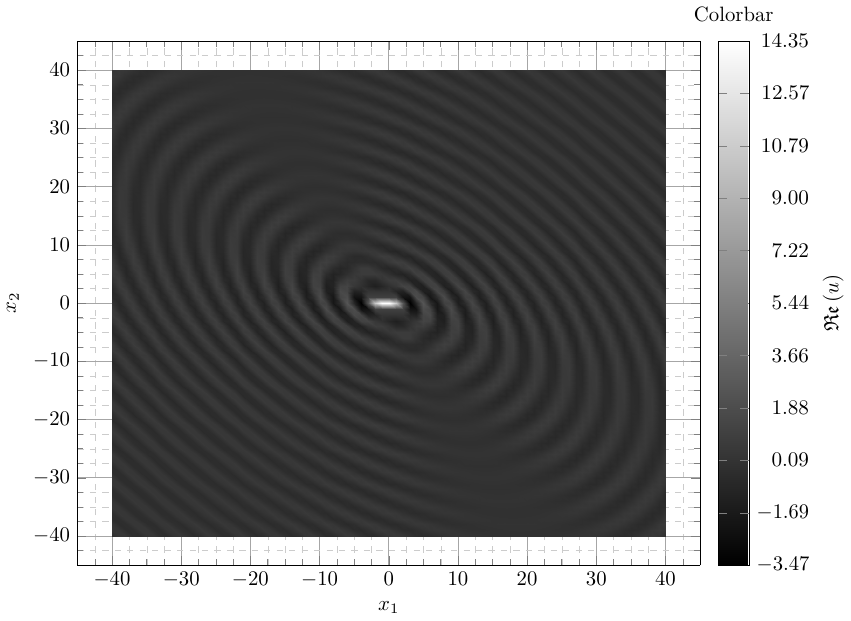}
        \caption{The density plot of $\Real{\cG}$; $\left( {x}_{1},{x}_{2} \right) \in {\bZ}^2$.}
        \label{fig:examp02_real_u_in_Z_sq}
    \end{subfigure}
    \hfill
    \begin{subfigure}{.49\textwidth}
        \centering
        \includegraphics[width=\textwidth,height=\textheight,keepaspectratio]{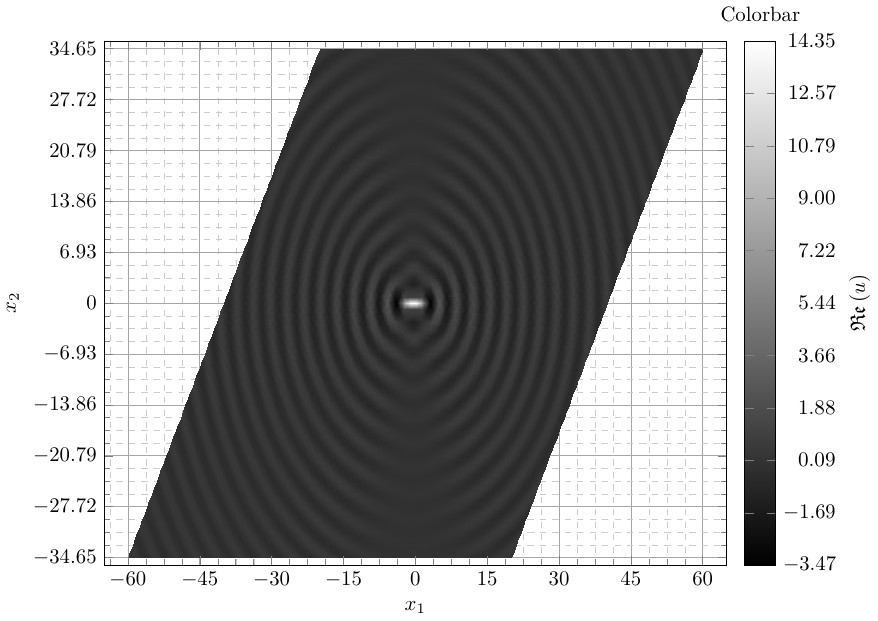}
        \caption{The density plot of $\Real{\cG}$; $\left( {x}_{1},{x}_{2} \right) \in {\bR}^2$.}
        \label{fig:examp02_real_u_in_R_sq}
    \end{subfigure}
    \caption{The plots of the real part of the lattice Green's functions in sub-figures (a) and (b) are presented in both the transformed and original coordinates of the triangular lattice, respectively, wherein $k = 2$.}
    \label{fig:examp02_real_u}
\end{figure}

\section{Discussion}\label{sec:discussion}

In this paper, we extend our investigation of the discrete Helmholtz equation and its associated exterior problems within a two-dimensional triangular lattice, focusing on two main questions: the effectiveness of the numerical methods used to evaluate the Green's function, including the case $k=2$, and the necessity of the cone condition. It is important to note that the computational accuracy of the solution representation formula \eqref{eq:sol1} for the discrete Helmholtz equation principally relies on the computation of lattice Green's functions. To illustrate these points numerically, we investigate two sample problems. The first sample problem involves four discrete boundary points arranged along a single line, while the second sample problem consists of ten discrete boundary points arranged on parallel line segments.

Various numerical quadrature techniques were attempted to handle the double integrals \eqref{eq:gmn} for ${k}^{2} + \imath\varepsilon$. In the sense of the existence of these integrals, instead of \eqref{eq:gmn} the following integrals \eqref{eq:pointwise} are considered (for details, see \hyperref[sec:Green's_represent_form]{Section \ref*{sec:Green's_represent_form}}). However, this approach turned out to be disadvantageous due to the rapid oscillations of the integrands as the magnitudes of ${x}_{1}$ and ${x}_{2}$ increase significantly from the origin. For the numerical computation of the lattice Green's functions, we employed a technique originated in the article \cite{Berciu2010}. Our numerical experiments demonstrate that this approach effectively approximates the lattice Green's functions, showing close agreement with physical solutions. Furthermore, this approach indicates stability properties when computations are initiated from a large Manhattan distance (refer to \hyperref[tab:Experim_Converg]{Table \ref*{tab:Experim_Converg}}).

In the second sample problem, despite not satisfying the cone condition, the uniqueness result is still attained. Numerical experiments revealed that as more boundary points are aligned on parallel lines (typically, the boundary points follow the configuration illustrated in \hyperref[fig:examp2_ten_points]{Figure \ref*{fig:examp2_ten_points}}), the determinant of the coefficient matrix $\cH$ of the linear system of boundary equations \eqref{eq:BS} approaches zero. In the given scenario, it is observed that $\left| \det\left( \cH \right) \right| \approx \pgfmathprintnumber[std,std=-1:0,fixed zerofill,sci zerofill,sci e,precision=4,1000 sep={}]{5.2308888861181847e-06}$. Concurrently, the condition number of the matrix $\cH$, with respect to the matrix norm induced by the (vector) Euclidean norm, is given by ${\kappa}_{2}\left( \cH \right) = {\left\| {\cH}^{-1} \right\|}_{2} {\left\| \cH \right\|}_{2} \approx \pgfmathprintnumber[std,std=-1:0,fixed zerofill,sci zerofill,sci e,precision=4,1000 sep={}]{1.533136475938260e+01}$. The fact that the determinant of the coefficient matrix $\cH$ is of the order of ${10}^{-6}$ suggests that the matrix is nearly singular. However, with a condition number less than $16$, it indicates that despite the near singularity of the matrix $\cH$, the system \eqref{eq:BS} remains relatively stable and well-behaved.

Another interesting question pertains to the structure of the space of radiating solutions. We focus our attention on the space ${\ell}_{\mathrm{R}}^{\infty}\left( \Omega \right)$, a Banach space comprising all bounded sequences on $\Omega \subset {\bZ}^{2}$ that satisfy the prescribed radiation condition \eqref{eq:radcond}. Depending on the specific objectives of the investigation, various spaces on lattices can be considered, as detailed in \cite{AndoIsozakiMorioka2016,ParraRichard2018}, which delves into the spectral properties of discrete Schr\"{o}dinger operators within lattice frameworks.

\section*{Acknowledgement}\label{sec:acknowledgement}
The authors wish to express their gratitude to Prof. Dr. Jemal Rogava for his helpful remarks concerning the numerical computation section of this article.

This work was supported by the Shota Rustaveli National Science Foundation of Georgia (SRNSFG) [grant number: FR-21-301, project title: ``Metamaterials with Cracks and Wave Diffraction Problems''].

\appendixpage
\begin{appendices}\label{append:main_sect}
	\section{Describing of assembling of sparse matrices}\label{append:sparse_matrix}
        We introduce a set of functions for constructing sparse matrices ${\alpha}_{n}\left( k \right)$, ${\beta}_{n}\left( k \right)$, and ${\gamma}_{n}\left( k \right)$, employing a language-agnostic approach. This procedure is delineated into two distinct algorithms: \hyperref[alg:matrix_assembly_part1]{Algorithm \ref*{alg:matrix_assembly_part1}} outlines the procedure for generating these sparse matrices when $n$ is even, while \hyperref[alg:matrix_assembly_part2]{Algorithm \ref*{alg:matrix_assembly_part2}} delineates the process for odd values of $n$. Specifically, the functions \textsc{Alpha2}$\left( p,k \right)$, \textsc{Beta2}$\left( p,k \right)$, and \textsc{Gamma2}$\left( p,k \right)$ in \hyperref[alg:matrix_assembly_part1]{Algorithm \ref*{alg:matrix_assembly_part1}} correspond to the computation of sparse matrices ${\alpha}_{2p}\left( k \right)$, ${\beta}_{2p}\left( k \right)$, and ${\gamma}_{2p}\left( k \right)$, respectively. Furthermore, the functions \textsc{Alpha1}$\left( p,k \right)$, \textsc{Beta1}$\left( p,k \right)$, and \textsc{Gamma1}$\left( p,k \right)$ in \hyperref[alg:matrix_assembly_part2]{Algorithm \ref*{alg:matrix_assembly_part2}} clarify the computations of matrices ${\alpha}_{2p + 1}\left( k \right)$, ${\beta}_{2p + 1}\left( k \right)$, and ${\gamma}_{2p + 1}\left( k \right)$, respectively.

        In \hyperref[tab:size_and_nnz_sparse_matrices]{Table \ref*{tab:size_and_nnz_sparse_matrices}}, we present the size and the number of nonzero-valued elements of these sparse matrices for each case.
        
        \begin{table}[H]
		\centering


\begin{tabular}{ccc}
	\toprule
	\textbf{Matrix} & \textbf{Size} & \textbf{Number of nonzero elements} \\
	\midrule
        ${\alpha}_{2p}\left( k \right)$ & $\left( p + 1 \right) \times p$ & $2p$ \\
        ${\beta}_{2p}\left( k \right)$ & $\left( p + 1 \right) \times \left( p + 1 \right)$ & $2p + 1$ \\
        ${\gamma}_{2p}\left( k \right)$ & $\left( p + 1 \right) \times \left( p + 1 \right)$ & $3p + 1$ \\
        ${\alpha}_{2p + 1}\left( k \right)$ & $\left( p + 1 \right) \times \left( p + 1 \right)$ & $2p + 1$ \\
        ${\beta}_{2p + 1}\left( k \right)$ & $\left( p + 1 \right) \times \left( p + 2 \right)$ & $2p + 2$ \\
        ${\gamma}_{2p + 1}\left( k \right)$ & $\left( p + 1 \right) \times \left( p + 1 
	\right)$ & $3p + 1$ \\
	\bottomrule
\end{tabular}
		\caption{The dimensions of sparse matrices and their corresponding nonzero elements in computing lattice Green's functions.}
		\label{tab:size_and_nnz_sparse_matrices}
	\end{table}
        
	\begin{algorithm}
		\caption{Assembly of sparse matrices ${\alpha}_{2p}\left( k \right)$, ${\beta}_{2p}\left( k \right)$, and ${\gamma}_{2p}\left( k \right)$.}
		\label{alg:matrix_assembly_part1}
		\begin{algorithmic}[1]
			\Require $p$ (positive integer), $k$  \Comment{Input parameters: $p \geq 1$ and $k \in \left( 0,2\sqrt{2} \right)$.}
			\Statex
			\Function{Alpha2}{$p, k$}
			\State Initialize ${\alpha}_{2}$ as a $\left( p + 1 \right) \times p$ matrix with all entries set to zero.
			\For{$i \gets 1$ to $p$}
			\State ${\alpha}_{2}\left[ i,i \right] \gets 1$ \Comment{Set diagonal entries to $1$.}
			\If{$i \geq 2$}
			\State ${\alpha}_{2}\left[ i,i - 1 \right] \gets 1$ \Comment{Set sub-diagonal entries to $1$.}
			\EndIf
			\EndFor
			\State ${\alpha}_{2}\left[ p + 1, p \right] \gets 2$ \Comment{Set the $\left( p + 1, p \right)$ entry to $2$.}
			\State \textbf{return} ${\alpha}_{2}$ \Comment{Return the computed matrix ${\alpha}_{2}$.}
			\EndFunction
			\Statex
			\Function{Beta2}{$p, k$}
			\State Initialize ${\beta}_{2}$ as a $\left( p + 1 \right) \times \left( p + 1 \right)$ matrix with all entries set to zero.
			\State ${\beta}_{2}\left[ p + 1,p + 1 \right] \gets 2$ \Comment{Set the value of the bottom-right corner entry to $2$.}			
                \State ${\beta}_{2}\left[ 1,2 \right] \gets {\beta}_{2}\left[ p + 1,p + 1 \right]$ \Comment{Copy the value from the bottom-right corner entry to entry $\left( 1, 2 \right)$.}
                \For{$i \gets 1$ to $p$}
			\State ${\beta}_{2}\left[ i,i \right] \gets 1$ \Comment{Set diagonal entries to $1$.}
			\If{$i \geq 2$} \Comment{Check if $i$ is greater than or equal to $2$.}
			\State ${\beta}_{2}\left[ i,i + 1 \right] \gets 1$ \Comment{Set super-diagonal entries to $1$.}
			\EndIf
			\EndFor
			\State \textbf{return} ${\beta}_{2}$ \Comment{Return the computed matrix ${\beta}_{2}$.}
			\EndFunction
			\Statex
			\Function{Gamma2}{$p, k$}
			\State Initialize ${\gamma}_{2}$ as a $\left( p + 1 \right) \times \left( p + 1 \right)$ matrix with all entries set to zero.
			\State ${\gamma}_{2}\left[ 1,2 \right] \gets -2$ \Comment{Set the $\left( 1, 2 \right)$ entry to $-2$.}			
                \State ${\gamma}_{2}\left[ p + 1,p \right] \gets {\gamma}_{2}\left[ 1,2 \right]$ \Comment{Copy the value from entry $\left( 1, 2 \right)$ to entry $\left( p+1, p \right)$.}
                \For{$i \gets 1$ to $p + 1$}
			\State ${\gamma}_{2}\left[ i,i \right] \gets 6 - k^2$ \Comment{Set diagonal entries to $6 - k^2$.}
			\If{$i \geq 2$ \textbf{and} $i \leq p$} \Comment{Set the off-diagonal elements if applicable.}
			\State ${\gamma}_{2}\left[ i,i + 1 \right] \gets -1$ \Comment{Set super-diagonal entries to $-1$.}
			\State ${\gamma}_{2}\left[ i,i - 1 \right] \gets -1$ \Comment{Set sub-diagonal entries to $-1$.}
			\EndIf
			\EndFor
			\State \textbf{return} ${\gamma}_{2}$ \Comment{Return the computed matrix ${\gamma}_{2}$.}
			\EndFunction
		\end{algorithmic}
	\end{algorithm}
	\clearpage
	\begin{algorithm}
		\caption{Assembly of sparse matrices ${\alpha}_{2p + 1}\left( k \right)$, ${\beta}_{2p + 1}\left( k \right)$, and ${\gamma}_{2p + 1}\left( k \right)$.}
		\label{alg:matrix_assembly_part2}
		\begin{algorithmic}[1]
			\Require $p$ (non-negative integer), $k$ \Comment{Input parameters: $p \geq 0$ and $k \in \left( 0,2\sqrt{2} \right)$.}
			\Statex
			\Function{Alpha1}{$p, k$}
			\State Initialize ${\alpha}_{1}$ as a $\left( p + 1 \right) \times \left( p + 1 \right)$ matrix with all entries set to zero.
			\For{$i \gets 1$ to $p + 1$}
			\State ${\alpha}_{1}\left[ i,i \right] \gets 1$  \Comment{Set diagonal entries to $1$.}
			\If{$i \geq 2$} \Comment{Set the sub-diagonal elements if applicable.}
			\State ${\alpha}_{1}\left[ i,i - 1 \right] \gets 1$ \Comment{Set sub-diagonal entries to $1$.}
			\EndIf
			\EndFor
			\State \textbf{return} ${\alpha}_{1}$ \Comment{Return the computed matrix ${\alpha}_{1}$.}
			\EndFunction
			\Statex
			\Function{Beta1}{$p, k$}
			\State Initialize ${\beta}_{1}$ as a $\left( p + 1 \right) \times \left( p + 2 \right)$ matrix with all entries set to zero.
			\State ${\beta}_{1}\left[ 1,2 \right] \gets 2$  \Comment{Set the $\left( 1, 2 \right)$ entry to $2$.}
			\For{$i \gets 1$ to $p + 1$}
			\State ${\beta}_{1}\left[ i,i \right] \gets 1$ \Comment{Set diagonal entries to $1$.}
			\If{$i \geq 2$} \Comment{Set the super-diagonal elements if applicable.}
			\State ${\beta}_{1}\left[ i,i + 1 \right] \gets 1$  \Comment{Set super-diagonal entries to $1$.}
			\EndIf
			\EndFor
			\State \textbf{return} ${\beta}_{1}$ \Comment{Return the computed matrix ${\beta}_{1}$.}
			\EndFunction
			\Statex
			\Function{Gamma1}{$p, k$}
			\State Initialize ${\gamma}_{1}$ as a $\left( p + 1 \right) \times \left( p + 1 \right)$ matrix with all entries set to zero.
			\If{$p = 0$} \Comment{Consider the case when ${\gamma}_{1}$ is a $1 \times 1$ matrix.}
                \State ${\gamma}_{1}\left[ p + 1,p + 1 \right] \gets 4 - k^2$ \Comment{Set the only entry at position $\left( 1,1 \right)$ to $4 - k^2$.}
                \Else \Comment{For non-zero values of $p$.}
			\State ${\gamma}_{1}\left[ p + 1,p + 1 \right] \gets 5 - k^2$ \Comment{Set the $\left( p + 1,p + 1 \right)$ entry to $5 - k^2$.}
			\State ${\gamma}_{1}\left[ 1,2 \right] \gets -2$ \Comment{Set the $\left( 1, 2 \right)$ entry to $-2$.}
			\State ${\gamma}_{1}\left[ p + 1,p \right] \gets -1$ \Comment{Set the $\left( p + 1,p \right)$ entry to $-1$.}
			\For{$i \gets 1$ to $p$}
			\State ${\gamma}_{1}\left[ i,i \right] \gets 6 - k^2$ \Comment{Set diagonal entries to $6 - k^2$.}
			\If{$i \geq 2$}  \Comment{Set the off-diagonal elements if applicable.}
			\State ${\gamma}_{1}\left[ i,i + 1 \right] \gets -1$ \Comment{Set super-diagonal entries to $-1$.}
			\State ${\gamma}_{1}\left[ i,i - 1 \right] \gets -1$ \Comment{Set sub-diagonal entries to $-1$.}
			\EndIf
			\EndFor
			\EndIf
			\State \textbf{return} ${\gamma}_{1}$ \Comment{Return the computed matrix ${\gamma}_{1}$.}
                \EndFunction
		\end{algorithmic}
	\end{algorithm}
\end{appendices}

\def\printchapternonum{}
\bibliographystyle{plainnat}
\bibliography{bibsource}

\end{document}